\documentclass[oneside]{amsart}

\usepackage{cite}
\usepackage{hyperref}

\usepackage{amssymb}
\usepackage{mathtools}

\usepackage[noabbrev]{cleveref}
\usepackage{bbm}
\usepackage{enumerate}
\usepackage{thmtools}

\newtheorem{theorem}{Theorem}[section]
\newtheorem{lemma}[theorem]{Lemma}
\newtheorem{proposition}[theorem]{Proposition}
\newtheorem{corollary}[theorem]{Corollary}
\newtheorem{assumption}[theorem]{Assumption}
\newtheorem*{extra*}{Proof of \Cref{prop:box}}
\theoremstyle{remark}
\newtheorem{remark}[theorem]{Remark}

\newcommand{\maxIndex}{i^{(1)}_t}
\newcommand{\secIndex}{i^{(2)}_t}
\newcommand{\maxPhi}{\varphi^{(1)}_t}
\newcommand{\maxZ}{z^{(1)}_t}
\newcommand{\secZ}{z^{(2)}_t}
\newcommand{\maxLambda}{\lambda^{(1)}_t}
\newcommand{\PP}{\mathbb{P}}
\newcommand{\EE}{\mathbb{E}}
\newcommand \id{\mathbbm 1}
\newcommand{\ballOfExpDecay}[1]{B( #1, h_t |#1| /3)}
\newcommand{\xDomain}[1]{\ballOfExpDecay{#1} \cup \{0\}}

\DeclareMathOperator*{\argmax}{arg\,max}

\DeclarePairedDelimiter\floor{\lfloor}{\rfloor}

\begin{document}
\author{Artiom Fiodorov \and Stephen Muirhead  \\
\tiny{University College London}  \\
\tiny \lowercase{a.fiodorov@ucl.ac.uk  \qquad s.muirhead@ucl.ac.uk}
 }

\title[Complete localisation in the PAM with Weibull potential]{Complete localisation and exponential shape of the  parabolic Anderson model with Weibull potential field}

\begin{abstract}
We consider the parabolic Anderson model with Weibull potential field, for all values of the Weibull parameter. We prove that the solution is eventually localised at a single site with overwhelming probability (complete localisation) and, moreover, that the solution has exponential shape around the localisation site. We determine the localisation site explicitly, and derive limit formulae for its distance, the profile of the nearby potential field and its ageing behaviour. We also prove that the localisation site is determined locally, that is, by maximising a certain time-dependent functional that depends only on: (i) the value of the potential field in a neighbourhood of fixed radius around a site; and (ii) the distance of that site to the origin.

Our results extend the class of potential field distributions for which the parabolic Anderson model is known to completely localise; previously, this had only been established in the case where the potential field distribution has sub-Gaussian tail decay, corresponding to a Weibull parameter less than two.
\end{abstract}


\thanks{This research was supported by a Graduate Research Scholarship from University College London and the Leverhulme Research Grant RPG-2012-608 held by Nadia Sidorova. We gratefully acknowledge the extensive feedback provided by Nadia Sidorova.}

\date{\today}

\maketitle

\vspace{-0.5cm}
\subsection*{Correction to published version} 
This is an updated version of \cite{Fiodorov14}, correcting an error in the definition of the $n$\textit{-local principal eigenvalue} $\tilde{\lambda}_t^{(n)}(z)$ used to define the localisation site $Z_t^{(1, \rho)}$. This error does not affect the conclusions of the paper.

In \cite{Fiodorov14}, $\tilde{\lambda}_t^{(n)}(z)$ is defined as the principal eigenvalue of the operator
\begin{align*}
\tilde{\mathcal{H}}_n^{(z)} &:=  \left(\Delta_{V_t} +  \tilde{\xi} +  (\xi - \tilde{\xi}) \id_{\{z\}} \right)\id_{B(z, n)} \\ &= \Delta_{V_t}\id_{B(z, n)}  +   \left( \tilde{\xi}  +  (\xi - \tilde{\xi}) \id_{\{z\}}  \right) \id_{B(z, n)}  \, .
\end{align*}
While this is sufficient in the case $\gamma < 4$, the key Lemma \ref{lem:jtorho} does not hold under this definition if $\gamma \ge 4$. Instead, in the general case $\tilde{\lambda}_t^{(n)}(z)$ should be defined as the principal eigenvalue of the operator
\begin{align*}
\tilde{\mathcal{H}}_n^{(z)} :=  \Delta_{V_t}\id_{B(z, \hat n)}   +   \left( \tilde{\xi}  +  (\xi - \tilde{\xi}) \id_{\{z\}}  \right) \id_{B(z, n)}  \, ,
\end{align*}
where
\[   \hat n := \begin{cases}
n & \gamma < 4 \, ,\\
n + 1 & \gamma \ge 4 \, .
\end{cases} \]
Under this modification, Lemma \ref{lem:jtorho} and its proof, as well as the rest of the paper, remain valid unchanged (with appropriate modifications to notation). 
\section{Introduction}
\label{sec:intro}
\subsection{The parabolic Anderson model}
We consider the Cauchy equation on the lattice $(\mathbb{Z}^d, |\cdot|_{\ell^1})$
\begin{align} \label{PAM}
\frac{\partial u(t, z)}{\partial t} &= (\Delta + \xi) u(t, z) \, , \qquad (t, z) \in [0, \infty) \times \mathbb{Z}^d\\
\nonumber u(0, z) &= \id_{\{0\}}(z) \, , \, \quad \qquad \qquad z \in \mathbb{Z}^d
\end{align}
where $\Delta$ is the \textit{discrete Laplacian} on $\mathbb{Z}^d$ defined by $(\Delta f)(z) = \sum_{y\sim z} f(y)$, the set $\{\xi(z)\}_{z \in \mathbb{Z}^d }$ is a collection of independent identically distributed (i.i.d.) random variables known as the \textit{random potential field}, and $\id_{\{0\}}$ is the indicator function of the origin. For a large class of distributions $\xi(\cdot)$, \cref{PAM} has a unique non-negative solution (see \cite{Gartner90}). 

Equation (\ref{PAM}) is often called the \textit{parabolic Anderson model} (PAM), named after the physicist P.W.\ Anderson who used the random Schr\"{o}dinger operator $\mathcal{\bar{H}} := \Delta + \xi$ to model electron localisation inside a semiconductor (\textit{Anderson localisation}; see \cite{Anderson58}). The Cauchy form of the problem in \cref{PAM} arises naturally in a system consisting of a single particle undergoing diffusion while branching at a rate determined by a (random) potential field (see \cite{Gartner90}[Section 1.2]).

The PAM and its variants are of great interest in the theory of random processes because they exhibit \textit{intermittency}, that is, unlike other commonly studied random processes such as diffusions, their long-term behaviour cannot be described with an averaging principle. The PAM is said to \textit{localise} if, as $t \to \infty$, the \textit{total mass} of the process $ U(t) := \sum_{z \in \mathbb{Z}^d} u(t, z)$ is eventually concentrated on a small number of sites, i.e.\ if there exists a (random) \textit{localisation set} $\Gamma_t$ such that
\begin{align}
\label{eq:local}
\frac{\sum_{z \in \Gamma_t} u(t, z)}{U(t)} \to 1 \, \qquad \text{in probability} \,.
\end{align}
The most extreme form of localisation is \textit{complete localisation}, which occurs if the total mass is eventually concentrated at just one site, i.e.\ if $\Gamma_t$ can be chosen in \cref{eq:local} such that $|\Gamma_t| = 1$.

It turns out that complete localisation cannot hold almost surely, since the localisation site will switch infinitely often and so, at certain times, the solution must be concentrated on at least two distinct sites (see, e.g., \cite{Konig09} for an example of almost sure convergence in the PAM on exactly two sites). 

Note that elsewhere in the literature (see, e.g., \cite{Sidorova12}) the convention
$(\Delta f)(z) := \sum_{y\sim z} (f(y) - f(z))$ is used to define the discrete
Laplacian in the PAM\@. This is equivalent to shifting the random potential field by the constant $2d$, and makes no qualitative difference to the model.

\subsection{Localisation classes}
It is known that the strength of intermittency and localisation in the PAM is governed by the thickness of the upper-tail of the potential field distribution $\xi(\cdot)$, and in particular the asymptotic growth rate of 
$$ g_\xi(x) := -\log (\PP(\xi(\cdot) > x)) \,.$$ 
Depending on this growth rate, the PAM can exhibit distinct types of localisation behaviour, which are often categorised along two qualitative dimensions: (1) the number of connected components of $\Gamma_t$ (\textit{localisation islands}) in the limit (i.e.\ single, bounded or growing);
and (2) the size of each localisation island in the limit (i.e.\ single, bounded or growing). 

Universality classes with respect to the \textit{size} of each localisation island are well-understood (see, e.g., \cite{vanDerHofstad06} and \cite{Gartner07}). It was proven in \cite{Gartner07} that the double-exponential distribution forms the critical threshold between these classes. More precisely, if $g_\xi(x) = O(e^{x^\chi})$ for some $\chi < 1$ (i.e.\ tails heavier than double-exponential) then localisation islands consist of a single site. This class includes \textit{Weibull-like} tails, where $g_\xi(x) \sim x^\gamma$ for $\gamma > 0$, and \textit{Pareto-like} tails, where $g_\xi(x) \sim \gamma \log x$ for $\gamma > d$ (recall from \cite{Gartner90} that if $\gamma < d$ then the solution to \cref{PAM} is not well-defined; if $\gamma = d$ then the solution is well-defined only for $d > 1$). Conversely, if $e^{x^\chi} = O(g_\xi(x))$ for some $\chi > 1$ (i.e.\ tails lighter than double-exponential, including bounded tails), then the size of localisation islands grow to infinity. 

On the other hand, universality classes with respect to the \textit{number} of localisation islands are not at all well-understood. In particular, it is not known whether the PAM with $g_\xi(x) = O(e^{x^\chi})$ for some $\chi < 1$ \textit{always} exhibits complete localisation. Indeed, this was conjectured to be false in \cite{Konig09}. Up until now, complete localisation has only been exhibited for the PAM with Pareto potential (in \cite{Konig09}) and Weibull potential with parameter $\gamma < 2$ (in \cite{Sidorova12}), which includes the case of exponential tails. This has left open the question as to whether the PAM with Weibull potential with parameter $\gamma \geq 2$, which includes the important class of normal tails, also exhibits complete localisation.

\subsection{Main results}
We consider the PAM with Weibull potential, that is, where $\xi(\cdot)$ satisfies $g_\xi(x) = x^\gamma$, for some $\gamma > 0$. We prove that the PAM with Weibull potential is eventually localised at a single site with overwhelming probability
(\textit{complete localisation}) and, moreover, that the renormalised solution has exponential shape around this site. We determine the localisation site explicitly, and derive limit formulae for its distance, the profile of the nearby potential field and its ageing behaviour. We also prove that the localisation site is determined locally, that is, by maximising a certain time-dependent functional that depends only on: (i) the values of $\xi(\cdot)$ in a neighbourhood of \textit{fixed} radius $\rho := \floor{(\gamma-1)/2}^+$ around a site, where $x^+:= \max\{x, 0\}$; and (ii) the distance of that site to the origin. In particular, if $\gamma < 3$ then $\rho = 0$ and so the localisation site is determined only by maximising a certain time-dependent functional of the pair $(\xi(\cdot), |\cdot|_{\ell^1})$. We shall refer to $\rho$ as the \textit{radius of influence}.

In order to state these results explicitly, we introduce some notation. Define a large `macrobox' $V_t := [-R_t, R_t]^d \subseteq \mathbb{Z}^d$, with $R_t := t (\log t)^{\frac{1}{\gamma}}$, identifying its opposite faces so that it is
properly considered a $d$-dimensional torus. Further, for each $a \leq 1$, define the associated macrobox level $L_{t, a} := ((1-a) \log |V_t|)^{\frac{1}{\gamma}}$ and let the subset $\Pi^{(L_{t,a})} := \left\{ z \in   V_t : \xi(z) > L_{t,a} \right\}$ consist of sites within the macrobox $V_t$ at which $\xi$-exceedences of the level $L_{t,a}$ occur. Define also, for each $z \in V_t$ and $n \in \mathbb{N}$, the ball $B(z, n) := \{y \in V_t: |y-z|_{\ell^1} \leq   n\}$, considered as a subset of $V_t$ (i.e.\ with the metric acting on the torus). Henceforth, for simplicity, we simply write $|\cdot|$ in place of $|\cdot|_{\ell^1}$ when denoting distances on $\mathbb{Z}^d$ or $V_t$. 

Fix a constant $0 < \theta < 1/2$, and abbreviate $L_{t}:= L_{t, \theta}$. Let $\tilde{\xi} := \xi \id_{V_t \setminus \Pi^{(L_t)}}$ be the \textit{$L_t$-punctured} potential field. For each $z \in V_t$ and $n \in \mathbb{N}$, define the $L_t$-punctured $n$-local Hamiltonian $\tilde{\mathcal{H}}_n^{(z)}$
\begin{align}
\label{eq:defH}
\tilde{\mathcal{H}}_n^{(z)} :=  \Delta_{V_t}\id_{B(z, \hat n)}  +  \left( \tilde{\xi} +  (\xi - \tilde{\xi}) \id_{\{z\}} \right) \id_{B(z, n)} \, ,
\end{align}
where $\Delta_{V_t}$ denotes $\Delta$ restricted to the torus $V_t$, and
\[   \hat n := \begin{cases}
n & \gamma < 4 \, , \\
n + 1 & \gamma \ge 4 \, .
\end{cases} \]
Let $\tilde{\lambda}_t^{(n)}(z)$ denote the principal eigenvalue of $\tilde{\mathcal{H}}_n^{(z)}$. To be clear, equation \eqref{eq:defH} means that $\tilde{\mathcal{H}}_n^{(z)}$ acts as
$$ (\tilde{\mathcal{H}}_n^{(z)} f)(x) =  \begin{cases}
\xi(x) f(x) + \sum_{ \{ y \in V_t: |y-x| = 1 \}} f(y)  &  \text{if } x \in \{z\} \cup (B(z, n) \setminus \Pi^{(L_t)}) \\
\sum_{ \{ y \in V_t: |y-x| = 1 \}} f(y)  & \text{if } x \in B(z, n) \cap (\Pi^{(L_t)} \setminus \{z\}) \\
 \sum_{ \{ y \in V_t: |y-x| = 1 \}} f(y)  & \text{if } x \in B(z, \hat n) \setminus B(z, n) \\
0 & \text{if } x \notin B(z, \hat n) \\
 \end{cases}$$
with all distances being on the torus $V_t$. 

We shall call $\tilde{\lambda}_t^{(n)}(z)$ the $n$-\textit{local principal eigenvalue at $z$} and remark that it is a certain function of the set $\xi^{(n)}(z) := \{\xi(y)\}_{y \in B(z, n)}$. Note that the $\{\tilde{\lambda}_t^{(n)}(z)\}_{z \in V_t}$ are identically distributed, and have a dependency range bounded by $2n$, i.e.\ the random variables $\tilde{\lambda}_t^{(n)}(y)$ and $\tilde{\lambda}_t^{(n)}(z)$ are independent if and only if $|y-z| > 2n$. Remark also that in the case $\gamma < 4$, $\tilde{\lambda}_t^{(0)}(z)$ is simply the potential $\xi(z)$.

For any sufficiently large $t$, define a \textit{penalisation functional} $\tilde{\Psi}^{(n)}_{t}: V_t \to \mathbb{R}$ by
$$ \tilde{\Psi}^{(n)}_{t}(z) := \tilde{\lambda}_t^{(n)}(z) - \frac{|z|}{\gamma t} \log \log t $$
and let $Z_{t}^{(1, n)} := \argmax_{z \in V_t} \tilde{\Psi}^{(n)}_{t}(z)$ and $T^{(n)}_t := \inf\{ s > 0 : Z_{t+s}^{(1, n)} \neq Z_t^{(1, n)} \}$. Note that, for any $t$, the site $Z_t^{(1, n)}$ is well-defined almost surely, since $V_t$ is finite. Moreover, as we shall see, $Z_t^{(1, n)}$ will turn out to be independent of the choice of $\theta$. 

Define a function $q:\mathbb{N} \to [0, 1]$ by
$$ q(x) := \left(1 - \frac{2x}{\gamma - 1} \right)^+ $$
using the convention that $0/0 := 0$. Introduce the scales
$$ r_t := \frac{t (d \log t)^{\frac{1}{\gamma} - 1}}{\log \log t}  \ , \ a_t := (d \log t)^\frac{1}{\gamma} \quad \text{and} \quad d_t := \frac{1}{\gamma}(d \log t)^{\frac{1}{\gamma} - 1}$$
and an auxiliary scaling function $\kappa_t \to 0$ that decays arbitrarily slowly. Finally, let $B_t$ denote the ball $\{z \in \mathbb{Z}^d : |z-Z_t^{(1, \rho)}| < r_t \kappa _t\}$, considered as a subset of $\mathbb{Z}^d$.

Our main results can then be summarised by the following:

\begin{theorem}[Profile of the renormalised solution] \label{thm:main1}
As $t \to \infty$, the following hold:
\begin{enumerate}[(a)]
\item \label{thm:main1a}
For each $z \in B_t$ uniformly,
$$ \frac{ \log \left( \frac{u(t, z)}{U(t)}\right) }{ \frac{1}{\gamma}|z-Z_t^{(1, \rho)}|\log \log t }  \to -1 \qquad \text{in probability} \, ;$$ 
\item \label{thm:main1b}
Moreover,
$$ e^{t d_t \kappa_t} \sum_{z \notin B_t} \frac{u(t, z)}{U(t)}  \qquad \text{is
bounded in probability} \,.$$ 
\end{enumerate}
\end{theorem}
\begin{corollary}[Complete localisation]
\label{cor:cl}
 As $t \to \infty$,
 $$ \frac{u(t, Z_t^{(1, \rho)})}{U(t)} \to 1 \qquad \text{in probability} \,.$$
\end{corollary}

\begin{theorem}[Description of the localisation site]
\label{thm:main2}
As $t \to \infty$, the following hold:
\begin{enumerate}[(a)]
\item (Localisation distance)
\label{thm:main2a}
$$ \frac{Z_t^{(1, \rho)}}{r_t} \Rightarrow X \qquad \text{in law} $$
where $X$ is a random vector whose coordinates are independent and Laplace distributed random variables with absolute-moment one; 
\item (Local profile of the potential field) 
\label{thm:main2b} \\
For each $z \in B(Z_t^{(1, \rho)}, \rho )$ uniformly,
$$ \frac{\xi(z)}{a_t^{q(|z-Z_t^{(1, \rho)}|)}} \to 1 \qquad \text{in probability} \, ;$$
\item (Ageing of the localisation site)
\label{thm:main2c}
$$\frac{T^{(\rho)}_t}{t} \Rightarrow \Theta \qquad \text{in law} $$
where $\Theta$ is a nondegenerate almost surely positive random variable.
\end{enumerate}
\end{theorem}

\begin{corollary}[Ageing of the renormalised solution]
\label{cor:ageing}
For any sufficiently small $\varepsilon > 0$, as $t \to \infty$,
$$\frac{T^\varepsilon_t}{t} \Rightarrow \Theta \qquad \text{in law} $$
where 
$$ T^\varepsilon_t := \inf \left\{s > 0: \left| \frac{u(t, \cdot)}{U(t)} - \frac{u(t+s, \cdot)}{U(t+s)} \right|_{\ell^\infty} > \varepsilon \right\} $$
and $\Theta$ is the same almost surely positive random variable as in Theorem \ref{thm:main2}.
\end{corollary}

\begin{remark}
\label{remark:intuition}
The localisation site $Z_t^{(1, \rho)}$ is the maximiser of the penalisation functional $\tilde{\Psi}^{(\rho)}_{t}(z)$, which balances the magnitude of the $\rho$-local principal eigenvalue at a site with the distance of that site from the origin. Heuristically, this may be explained as the solution favouring sites with high local principal eigenvalue but being `penalised' for diffusing too quickly. 

As claimed, $\tilde{\Psi}_t^{(\rho)}(z)$ depends only on the set $\xi^{(\rho)}(z)$ and on the distance $|z|$. Indeed, in order to determine $Z_t^{(1, \rho)}$ explicitly, a finite path expansion is available for $\tilde{\lambda}_t^{(\rho)}(z)$ (see \Cref{prop:pathexpforlambda} for a precise formulation):
\begin{align*}
\tilde{\lambda}_t^{(\rho)}(z) = \xi(z) + \sum_{2 \leq k \leq 2j} \sum_{\Gamma^\ast_{k}(z, \hat \rho)} \prod_{0 < i < k} \frac{1}{\tilde{\lambda}_t^{(\rho)}(z) - \tilde{\xi} \id_{B(z, \rho) }(y_i) } + o(d_t)  
\end{align*}
where $j :=
 [\gamma/2] \in \{\rho, \rho+1\} $  and $\Gamma^\ast_{k}(z, \hat \rho)$ is the set of all length $k$ nearest neighbour paths  
$$ z =: y_0 \to y_1 \to \ldots \to y_{k} := z \quad \text{in } B(z, \hat \rho)$$
such that $y_i \neq z$ for all $0 < i < k$. This path expansion can be iteratively evaluated to get an expression for $\tilde{\lambda}_t^{(\rho)}(z)$ as an explicit function of $\xi^{(\rho)}(z)$. Note that $j$ is chosen precisely to be the smallest non-negative integer such that $a_t^{-2j-1} = o(d_t)$, which ensures that paths with more than $2j$ steps contribute at most $o(d_t)$ to the sum.  Since we show in \Cref{sec:pp} that the gap between the maximisers of $\tilde{\Psi}^{(\rho)}_t$ is on the scale $d_t$, such an expression is sufficient to determine $Z_t^{(1, \rho)}$.
\end{remark}

\begin{remark}
Our limit theorem for the profile of the renormalised solution  holds within a distance $r_t \kappa_t$ of the localisation site, where $\kappa_t$ may be chosen to decay arbitrarily slowly. At or beyond this scale, the profile will be interrupted by `bumps' in the renormalised solution around other high values of the functional $\tilde{\Psi}^{(\rho)}_{t}$, which occur at distances on the scale $r_t$. In this region, we simply bound the renormalised solution by the height of these bumps, although we also expect a weaker global exponential decay to hold.
\end{remark}

\begin{remark}
The ageing of the renormalised solution in Corollary \ref{cor:ageing} is a natural consequence of complete localisation of the renormalised solution (Corollary \ref{cor:cl}) and the ageing of the localisation site (Theorem \ref{thm:main2}). The proof of this result is essentially the same as in \cite{Morters11}[Proposition 2.1] for the corresponding result in the case of Pareto potential field; we defer to that paper for the proof. Note also that Corollary \ref{cor:ageing} is a quenched ageing result along the lines of \cite{Morters11}, as opposed to the annealed (i.e.\ averaged over all realisations of the random environment) ageing studied in \cite{Gartner11}.
\end{remark}

\begin{remark}
Recall that it was previously shown in \cite{Sidorova12} that complete localisation holds in the case $\gamma < 2$. The analysis in that paper is broadly similar to ours, but uses the penalisation functional
$$ \Psi_t^{\ast}(z) := \xi(z) - \frac{|z|}{\gamma t} \log \log t $$ 
which equals $\tilde{\Psi}^{(\rho)}_t(z)$ in the special case $\gamma < 3$. This restricts the validity of the analysis to where there is an exact correspondence between the top order statistics of the fields $\xi$ and $\tilde{\lambda}_t^{(\rho)}$ in $V_t$. Clearly this holds for $\gamma < 3$ by definition. On the other hand, the exact correspondence has been shown to be false if $\gamma \geq 3$ (in \cite{Astrauskas13}), and so an analysis based on the functional $\Psi_t^\ast$ fails in that case.
\end{remark}

\begin{remark}
We briefly mention the strong possibility that our results can be extended to the case of fractional-double-exponential potential field, i.e.\ where $g_\xi(x) = e^{x^\chi}$ for some $\chi < 1$. The main difference in that case is that the radius of influence $\rho$ grows with $t$, which presents a technical difficulty in extending the results in \Cref{prop:asympt}. Nevertheless, we strongly believe such an extension is valid, and since the rest of our proof holds essentially unchanged, we expect complete localisation to also hold in the fractional-double-exponential case.
\end{remark}

The paper is organised as follows. In \Cref{sec:outline} we give an outline of the proof, and establish \Cref{thm:main1} subject to an auxiliary \Cref{thm:aux}. In \Cref{sec:prelim} we establish some preliminary results. In \Cref{sec:pp} we use a point process approach to study the random variables $Z_t^{(1, \rho)}$ and $\tilde{\Psi}^{(\rho)}_{t}(Z_t^{(1, \rho)})$ (and generalisations thereof), and in doing so complete the proof of \Cref{thm:main2}. In \Cref{sec:spec} we collect results from spectral theory that we will apply in \Cref{sec:completion}. In \Cref{sec:completion} we complete the proof of the auxiliary \Cref{thm:aux}.

\section{Outline of the Proof}
\label{sec:outline}
In the literature, the usual approach to study $u(t, \cdot)$ is with probabilistic methods via the Feynman-Kac representation (for instance, in \cite{Gartner07}). Our primary approach is different, applying spectral theory methods to the Hamiltonian $\mathcal{\bar{H}}$ (as is done in \cite{Astrauskas07}, for instance). We note, however, that these approaches are very similar, and we do at times make use of the Feynman-Kac representation.

\subsection{Spectral representation of the solution}
The basic idea that underlies our proof is that the solution $u(t, \cdot)$ is well-approximated by a spectral representation in terms of the eigenfunctions of the Hamiltonian $\mathcal{\bar{H}}$ restricted to a suitably chosen domain. It turns out that this spectral representation is asymptotically dominated by just one eigenfunction, which is eventually localised with exponential decay away from the localisation site.

In order to apply this idea, we restrict $\mathcal{\bar{H}}$ to the macrobox $V_t$ (i.e.\ with periodic boundary conditions, recalling that $V_t$ is a torus), on which the solution $u(t, \cdot)$ turns out to be essentially concentrated. So let $u_{V_t}(s, z)$ be the solution to the PAM restricted to $V_t$, that is, defined by the Hamiltonian $\mathcal{H} :=  \Delta_{V_t} + \xi$, with $u_{V_t}(s, z): = 0$ outside $V_t$ by convention, and let $U_{V_t}(t):= \sum_{z \in V_t} u_{V_t}(t, z)$.

\begin{proposition}[Correspondence between $u_{V_t}(t, z)$ and $u(t, z)$]\label{prop:box}
As $t \to \infty$ and for any $z$,
$$ |u_{V_t}(t, z) - u(t, z)| = o\left(e^{-R_t} \right) \qquad \text{and} \qquad |U_{V_t}(t) - U(t)| = o\left(e^{-R_t} \right) \, ,$$
where both hold almost surely.
\end{proposition}
\begin{remark}Since the error in \Cref{prop:box} is of lower order than the
  bounds in \Cref{thm:main1}, it will be sufficient to prove that
  \Cref{thm:main1} holds for $u_{V_t}(t, \cdot)$. \Cref{prop:box} is proved in
  \Cref{sec:prelim}.
\end{remark}

Denote by $\lambda_{t, i}$ and $\varphi_{t, i}$ the $i$'th largest eigenvalue and corresponding eigenvector of $\mathcal{H}$, with each $\varphi_{t, i}$ taken to be $\ell^2$-normalised with $\varphi_{t, i}(z):=0$ outside $V_t$ by convention. Since $V_t$ is bounded, the solution $u_{V_t}(t, \cdot)$ permits a spectral representation in terms of the eigenfunctions of $\mathcal{H}$: 
\begin{align}
\label{def:specrep}
u_{V_t}(t, \cdot) = \sum_{i=1}^{|V_t|} e^{t \lambda_{t, i} } \varphi_{t,i}(0) \varphi_{t,i}(\cdot) \,.
\end{align}
Define a functional $\Psi_{t}: \{1, 2, \ldots , |V_t|\} \to \mathbb{R} \cup \{-\infty\}$ by
$$ \Psi_{t}(i) := \lambda_{t, i} + \frac{\log |\varphi_{t, i}(0)|}{t} $$
and remark that this is chosen so that the magnitude of the $i$'th term in the sum in \cref{def:specrep} is $e^{t \Psi_{t}(i)} |\varphi_{t,i}(\cdot)|$, using the convention that $\exp\{ -\infty \} := 0$.

We refer to $\{\Psi_{t}(\cdot)\}$ as the \textit{penalised spectrum}, noting
that it represents a trade-off between the magnitude of the eigenvalue and the
(absolute) magnitude of the eigenvector at the origin; the intuition here is the
same as in \Cref{remark:intuition}. We prove that, with overwhelming probability, a gap exists between the largest two values in the penalised spectrum, which implies that the spectral representation in \cref{def:specrep} is dominated by just one eigenfunction. Moreover, we prove that this eigenfunction is eventually localised at $Z_t^{(1, \rho)}$. To make this precise, let $\maxIndex := \argmax_i \Psi_{t}(i)$ and $\secIndex := \argmax_{i \neq \maxIndex} \Psi_{t}(i) $, and abbreviate $\maxPhi := \varphi_{t, \maxIndex}$ and $\maxLambda := \lambda_{t, \maxIndex}$ for notational convenience. Moreover, introduce auxiliary scaling functions $f_t, h_t, e_t \to 0$ and $g_t \to \infty$ as $t \to \infty$ such that 
\begin{align*}
\max \{ 1 / \log \log t , \kappa_t \} \ll f_t h_t \ll f_t \ll h_t \ll e_t/g_t
\end{align*}
where $a_t \ll b_t$ is notational shorthand for $a_t = o(b_t)$.

\begin{theorem}[Auxiliary theorem]
\label{thm:aux}
As $t \to \infty$, the following hold:
\begin{enumerate}[(a)]
\item \label{thm:aux1}
(Gap in the penalised spectrum)
$$ \PP \left( \Psi_t(\maxIndex) - \Psi_t(\secIndex) > d_t e_t \right) \to 1 \, ;$$
\item \label{thm:aux2}
(Profile of the dominating eigenfunction)
\begin{enumerate}[(i)]
\item The sets $B_t$ and $V_t$ satisfy
$$ \PP(B_t \subseteq V_t ) \to 1 \, ;$$
\item For each $z \in B_t$ uniformly,
$$ \frac{ \log \maxPhi(z) }{\frac{1}{\gamma} |z - Z_t^{(1, \rho)}| \log \log t} \to - 1 \qquad \text{in probability} \, ;$$ 
\item Moreover,
\\ $$ e^{t d_t \kappa_t} \sum_{z \in V_t \setminus B_t} |\maxPhi(z)| \qquad \text{is bounded in probability} \,.$$ 
\end{enumerate}

\end{enumerate}
\end{theorem}
In \Cref{subsec:finish} immediately below we finish the proof of
\Cref{thm:main1} subject to the auxiliary \Cref{thm:aux}; the other sections of
the paper are dedicated to proving Theorems~\ref{thm:main2} and~\ref{thm:aux}. 

Our proof of \Cref{thm:aux} is based on the observation that $\Psi_t(i)$ is
asymptotically approximated by $\tilde{\Psi}^{(\rho)}_{t}(z_{t, i})$, where
$z_{t, i} := \argmax_{z} \varphi_{t, i}(z)$. This is useful, since it is simpler
to study the maximisers of $\tilde{\Psi}^{(\rho)}_{t}$ than it is to analyse
$\Psi_t(\maxIndex) - \Psi_t(\secIndex)$ directly. Using a point process
approach, we demonstrate a gap between the top two maximisers of
$\tilde{\Psi}^{(\rho)}_{t}$ (and generalisations thereof), and also describe the
location and the neighbouring potential field of the maximiser $Z_t^{(1, \rho)}$,
proving \Cref{thm:main2}. We then establish the validity of the approximation, which requires both a correspondence between eigenvalues and local principal eigenvalues, and an analysis of the decay of eigenfunctions, in particular finding bounds on the value of eigenfunctions at zero; here we draw heavily on the methods in \cite{Astrauskas07} and \cite{Astrauskas08}.

\subsection{Proof of \Cref{thm:main1} subject to the auxiliary \Cref{thm:aux}}
\label{subsec:finish}
Starting from the spectral representation in (\ref{def:specrep}), we pull out the term involving the maximising index $\maxIndex$, and bound the remainder in the $\ell^1$-norm:
\begin{align*}
\left| \frac{u_{V_t}(t, \cdot)}{e^{t \maxLambda} \maxPhi(0)} -  \maxPhi(\cdot) \right|_{\ell^1} & = \left| \sum_{\substack{i = 1 \\ \ \ i \neq \maxIndex}}^{|V_t|} \frac{e^{t \lambda_{t, i}} \varphi_{t, i}(0)}{e^{t \maxLambda} \maxPhi(0)} \varphi_{t,i}(\cdot) \right|_{\ell^1}
\\ & \leq  \sum_{\substack{i = 1 \\ \ \ i \neq \maxIndex}}^{|V_t|}  \exp \left\{ t \left( \Psi_{t}(i) -  \Psi_t(\maxIndex) \right) \right\} \left| \varphi_{t,i}(\cdot) \right|_{\ell^1}  \,.
\end{align*}
Bounding each $\left| \varphi_{t,i}(\cdot) \right|_{\ell^1}$ by the Cauchy-Schwarz inequality and each summand by the maximum gives
\begin{align*}
\left| \frac{u_{V_t}(t, \cdot)}{e^{t \maxLambda} \maxPhi(0)} -  \maxPhi(\cdot) \right|_{\ell^1} & \leq |V_t|^\frac{3}{2} \exp \left\{ t \left( \Psi_{t}(\secIndex) -  \Psi_t(\maxIndex) \right) \right\} \,.
\end{align*}
and so, applying part~(\ref{thm:aux1}) of \Cref{thm:aux}, eventually with overwhelming probability 
\begin{equation}
\label{eq:final1}
\left| \frac{u_{V_t}(t, \cdot)}{e^{t \maxLambda} \maxPhi(0)} - \maxPhi(\cdot) \right|_{\ell^1} < |V_t|^\frac{3}{2} \exp \left\{ -t d_t e_t \right\} \,.
\end{equation}
By the triangle inequality, this implies that
$$ \left| \frac{U_{V_t}(t)}{e^{t \maxLambda} \maxPhi(0)} - \sum_{z \in V_t} \maxPhi(z)\right| <  |V_t|^\frac{3}{2} \exp \{ -t d_t e_t \} $$
and so, applying part~(\ref{thm:aux2}) of \Cref{thm:aux} we have that
\begin{equation}
\label{eq:final2}
e^{t \maxLambda} \maxPhi(0) = U_{V_t}(t)(1 + o(1)) \,.
\end{equation}

Consider now any $z \in B_t$. Combining part~(\ref{thm:aux2}) of \Cref{thm:aux} with equations (\ref{eq:final1}) and (\ref{eq:final2}) we have that, with overwhelming probability
\begin{align*}
\frac{u_{V_t}(t, z)}{U_{V_t}(t)} &= \frac{u_{V_t}(t, z)}{e^{t \maxLambda} \maxPhi(0)}(1+o(1)) \\
&= \exp \left\{ -\frac{1}{\gamma} |z-k_t^{(1)}| \log \log t \, (1+o(1)) \right\}(1+o(1))
\end{align*}
where $o(1)$ does not depend on $z$, recalling that $|z-Z_t^{(1, \rho)}| \log \log t =
o(t d_t e_t)$ for $z \in B_t$ since $h_t = o(e_t)$. Remark that the
correspondence in \Cref{prop:box} implies that, for any $z$ and with overwhelming probability,
\begin{align*}
\left| \frac{u(t, z)}{U(t)} -  \frac{u_{V_t}(t, z)}{U_{V_t}(t)} \right| & \leq  \frac{1}{U(t)} \left(  |u(t, z) - u_{V_t}(t, z)| + \frac{u_{V_t}(t, z)}{U_{V_t}(t)} |U(t) - U_{V_t}(t)| \right) \\
&= o(\exp \{ -R_t \}) = o(\exp \{-t d_t e_t \} )\,.
\end{align*}
and so, putting these together, we have
$$ \log \left( \frac{u_{V_t}(t, z)}{U_{V_t}(t)} \right) =  -\frac{1}{\gamma}
|z-Z_t^{(1, \rho)}| \log \log t \, (1+o(1)) $$
where $o(1)$ does not depend on $z$, which proves part~(\ref{thm:main1a}) of \Cref{thm:main1}.

On the other hand, combining part~(\ref{thm:aux2}) of \Cref{thm:aux} with \Cref{prop:box} and equation (\ref{eq:final2}), we have that
\begin{align*}
e^{t d_t \kappa_t} \sum_{z \notin B_t} \frac{u(t, z)}{U(t)} &< e^{t d_t \kappa_t} \sum_{z \in V_t \setminus B_t} \frac{u(t, z)}{U(t)} + o(1) \\
 & < e^{t d_t \kappa_t } \left( \sum_{z \in V_t \setminus B_t} \frac{u_{V_t}(t, z)}{e^{t \maxLambda} \maxPhi(0)}\right)(1+o(1)) 
\end{align*}
which is bounded in probability. \Cref{thm:main1} is proved.
\qed

\section{Preliminaries}
\label{sec:prelim}
In this section we establish some preliminary results. Denote by $\xi_{t, i}$ the $i$'th highest value of $\xi$ in $V_t$.

\begin{lemma}[Almost sure asymptotics for $\xi$]
\label{lem:asforxi}
For any $a_1 \in [0,1)$ and $a_2 \in (0,1]$, 
$$\xi_{t, \floor{|V_t|^{a_1}}} \sim L_{t, a_1} \qquad \text{and} \qquad |\Pi^{(L_{t, a_2})}| \sim |V_t|^{a_2} $$
hold almost surely.
\end{lemma}
\begin{proof}
These follow from well-known results on sequences of i.i.d.\ random variables;
they are proved in a similar way as~\cite[Lemma 4.7]{vanDerHofstad08}.
\end{proof}

\begin{lemma}[Almost sure separation of high points; see {\cite[Lemma 1]{Astrauskas07}}]
\label{lem:assep}
For any $\varepsilon < \theta$, and for each $n \in \mathbb{N}$, eventually
$$r \left( \Pi^{(L_t)} \right) > |V_t|^{\frac{1-2\epsilon}{d}} > n$$
almost surely, where $ r\left(S \right) := \min_{x \neq y \in S} \{ |x - y| \} $.
\end{lemma}

\begin{lemma}[Bounds on principal eigenvalues]
\label{lem:minmax}
For each $n \in \mathbb{N}$ and $z \in V_t$,
$$  \xi(z) \le \tilde{\lambda}^{(n)}_t(z) \le \max\{L_t, \xi(z)\} + 2d $$
Moreover,
$$\lambda_{t, 1} \leq \xi_{t, 1} + 2d \,.$$
\end{lemma}
\begin{proof}
These follow from the min-max theorem for the principal eigenvalue.
\end{proof}
    
\begin{extra*}
\end{extra*}
Note that the weaker statement that $|U_{V_t}(t) - U(t)| \to 0$ is proved
in~\cite[Section 2.5]{Gartner98} (although for a slightly different macrobox);
we need to control the error more precisely.

For $z \in \mathbb{Z}^d$, let ${[z]}_{V_t}$ denote the site in $V_t$ that belongs
to the equivalence class of $z$ in the quotient space $\mathbb{Z}^d \setminus
V_t$. Further, define a field $\xi^{\text{per}}_{V_t}$ on $\mathbb{Z}^d$ by
$\xi^{\text{per}}_{V_t}(\cdot) := \xi({[\cdot]}_{V_t})$. For a fixed $t > 0$,
consider the Feynman-Kac representations of $u(t, z)$ and $u_{V_t}(t, z)$:
\begin{equation}
u(t, z) = \EE \left[ \exp \left\{ \int_0^t \xi(X_s) + 2d \, ds \right\} \id_{\{X_t = z\}} \right]
\label{def:fk1}
\end{equation}
\begin{equation}
u_{V_t}(t, z) =  \EE \left[ \exp\left\{ \int_0^t \xi^{\text{per}}_{V_t}(X_s) + 2d
\, ds \right\} \id_{\{{[X_t]}_{V_t} = z\}} \right]
\label{def:fk2}
\end{equation}
where ${\{X_s\}}_{s \in \mathbb{R}^+}$ denotes the continuous-time random walk
on the lattice $\mathbb{Z}^d$ based at the origin, $\id_{A}$ denotes the
indicator function for the event $A$, and where the expectation $\EE$ is taken
over the trajectories of the random walk $X_s$.

For each $n \in \mathbb{N}$, let $e_n(X)$ denote the event that $\max_{s<t}
|X_s|_{\ell^\infty} = n$. Let $u^n(t, z)$ and $u^n_{V_t}(t,z)$ denote, respectively,
the expectations in (\ref{def:fk1}) and (\ref{def:fk2}) restricted to the event
$e_n(X)$, and define $U^n(t) := \sum_{z \in \mathbb{Z}^d} u^n(t, z)$ and
$U^n_{V_t}(t) := \sum_{z \in \mathbb{Z}^d} u^n_{V_t}(t, z)$ by analogy with
$U(t)$ and $U_{V_t}(t)$ respectively. Then it is clear, for each $z$, that
\begin{align}
\label{eq:lessthanrt}
\sum_{n < R_t} u^n(t, z) =  \sum_{n < R_t} u_{V_t}^n(t, z) \,.
\end{align}
Further, if $\xi^{(n)}_{1}$ is the largest value of $\xi$ in the box $\{z \in \mathbb{Z}^d : |z|_{\ell^\infty} \le n \}$, then 
$$ \max \{ U^n(t), U^n_{V_t}(t) \} \leq e^{t (\xi^{(n)}_{1} + 2d)} \PP(e_n(X)) \,.$$
As $n \to \infty$, we can bound $\xi^{(n)}_{1} + 2d$ almost surely with \Cref{lem:asforxi}:
$$ \xi^{(n)}_{1} + 2d \sim {(d \log n)}^{\frac{1}{\gamma}} \,.$$
For $n \geq R_t$ and by Stirling's approximation, we can also bound the probability $\PP(e_n(X))$ by
$$ \log \PP(e_n(X)) \leq \log \text{Pn}_{2dt}(n)  < - n \log n + n\log t + O(n) $$
where $\text{Pn}_{a}(n)$ denotes the probability mass function for the Poisson distribution with mean $a$, evaluated at $n$. Combining these bounds, for $n \geq R_t$ and as $t \to \infty$ eventually
$$ \max \{ U^n(t), U^n_{V_t}(t) \} < \exp \{ t (d \log n)^{\frac{1}{\gamma}}(1+\varepsilon) - n \log n + n\log t + C n) \} $$
almost surely, for any $\varepsilon > 0$ and for some $C > 0$. Since $n \geq R_t = t (\log t)^{\frac{1}{\gamma}}$, for $t$ large enough this can be further bounded as
$$ \max \{ U^n(t), U^n_{V_t}(t) \} < \exp \{-(1-\varepsilon) n \log n \} \,.$$
This implies that, eventually
\begin{align}
\label{eq:morethanrt}
\sum_{n \geq R_t} \max \{ U^n(t), U^n_{V_t}(t) \} < e^{- (1-\varepsilon) R_t \log R_t } \sum_{n \geq 0} e^{- (1-\varepsilon) n \log R_t} = o\left( e^{-R_t} \right) 
\end{align}
holds almost surely. Combining equations (\ref{eq:lessthanrt}) and (\ref{eq:morethanrt}), we get that
\begin{align*}
 \left|u(t, z) - u_{V_t}(t, z) \right| &=  \left|\sum_{n \geq R_t} u^n(t, z) - u^n_{V_t}(t, z) \right| \leq \sum_{n \geq R_t} u^n(t, z) + u^n_{V_t}(t, z) \\
 &  \leq \sum_{n \geq R_t} U^n(t) + U^n_{V_t}(t) \leq 2 \sum_{n \geq R_t} \max \{ U^n(t), U^n_{V_t}(t) \} = o(e^{-R_t})
 \end{align*}
and, similarly,
\begin{align*}
|U(t) - U_{V_t}(t)| \leq \sum_{n \geq R_t} U^n(t) + U^n_{V_t}(t) = o(e^{-R_t})
\end{align*}
as required.
\qed

\section{A Point Process Approach}
\label{sec:pp}
In this section, we use point process techniques to study the random variables
$Z_{t}^{(1, \rho)}$ and $\tilde{\Psi}^{(\rho)}_{t}(Z_{t}^{(1, \rho)})$, and
generalisations thereof; the techniques used are similar to those found in
\cite{Sidorova12}. In the process, we complete the proof of \Cref{thm:main2}.

\subsection{Point process asymptotics}
Fix an $0 < \varepsilon < \theta$ and an $0 < \eta < 2\rho - \gamma + 3$, remarking that the latter is possible by the definition of $\rho$. Recall also the definition $j := [\gamma / 2] \in \{ \rho, \rho + 1\}$. For each $n \in \mathbb{N}$ such that $n \leq j$, define the annuli $\bar{B}_1:= B({0}, \min\{n, \rho\}) \setminus \{0\}$ and $\bar{B}_2 := B(0, j) \setminus (\bar{B}_1 \cup \{0\})$, and the following $|\bar{B}_1 \cup \bar{B}_2|$-dimensional rectangles:
\begin{align*}
E^{(n)} := E_1^{(n)} \times E_2^{(n)} := \prod_{y \in \bar{B}_1} (1 - f_t, 1 + f_t) \times \prod_{y \in \bar{B}_2 } (0, a_t^\eta)  
\end{align*} 
and, after rescaling $E_1^{(n)}$ in each dimension, 
$$ S^{(n)} := \prod_{y \in \bar{B}_1} a_t^{q(|y|)} \pi_y (E_1^{(n)}) \times E_2^{(n)} $$
where $\pi_y$ is the projection map with respect to $y$. Finally, for each $z \in V_t$, define the event
\begin{align*}
\mathcal{S}^{(n)}_t(z) := \{ \xi(z) \in a_t (1-f_t, 1+f_t) \} \cup \big\{ \{\xi(z+y)\}_{y \in \bar{B}_1 \cup \bar{B}_2} \in S^{(n)} \big\} 
\end{align*}
with $\bar{\mathcal{S}}^{(n)}_t(z)$ its complement.

\begin{proposition}[Path expansion for $\tilde{\lambda}^{(n)}_t$]
\label{prop:pathexpforlambda}
As $t \to \infty$, for each $n \in \mathbb{N}$ and $z \in \Pi^{(L_{t, \varepsilon})}$ uniformly,
\begin{align*}
\tilde{\lambda}_t^{(n)}(z)  &=  \xi(z) + \sum_{k \geq 2} \sum_{\Gamma^\ast_{k}(z, \hat n)} \prod_{0 < i < k} \frac{1}{\tilde{\lambda}_t^{(n)}(z) - \xi \id_{B(z, n)}(y_i)}  \\
& = \xi(z) + \sum_{2 \leq k \leq 2j} \sum_{\Gamma^\ast_{k}(z, \hat n)} \prod_{0 < i < k} \frac{1}{\tilde{\lambda}_t^{(n)}(z) - \xi \id_{B(z, n)} (y_i)} + o(d_t e_t)  
\end{align*}
almost surely, where $\Gamma^\ast_{k}(z, \hat n)$ is the set of all length $k$ nearest neighbour paths 
$$ z =: y_0 \to y_1 \to \ldots \to y_{k} := z \quad \text{in } B(z, \hat  n)$$
such that $y_i \neq z$ for all $0 < i < k$.
\end{proposition}

\begin{proof}
As in \cite[Lemma 2]{Astrauskas07}, the eigenvalue $\tilde{\lambda}^{(n)}_t(z)$ satisfies
\begin{align}
\label{eq:pathexpl}
\nonumber \frac{1}{\xi(z)} &= \sum_{k \geq 0} \sum_{\Gamma_{k}(z, \hat n)} \prod_{0 \leq
i \le k} \frac{1}{\tilde{\lambda}_t^{(n)}(z) - \tilde{\xi}\id_{B(z, n)} (y_i)}   \\
& =  \frac{1}{\tilde{\lambda}^{(n)}_t(z)}   \sum_{k \geq 0} \sum_{\Gamma_{k}(z, \hat n)} \prod_{0 \le
i < k} \frac{1}{\tilde{\lambda}_t^{(n)}(z) - \tilde{\xi} \id_{B(z, n)} (y_i)}  
\end{align}
where $\Gamma_{k}(z, \hat n)$ is the set of all length $k$ nearest neighbour paths 
$$ z =: y_0 \to y_1 \to \ldots \to y_{k} := z \quad \text{in } B(z, \hat n) $$
i.e.\ including paths that return to $z$ multiple times; $\Gamma_{0}(z, \hat n)$ is understood to consist of a single degenerate path. Remark that the factor $1/\tilde{\lambda}^{(n)}_t(z)$ in equation \eqref{eq:pathexpl} appears since $\tilde{\xi}(y_k) = \tilde{\xi}(z) = 0$. Noticing that, by \Cref{lem:minmax}, $\tilde{\lambda}^{(n)}_t(z) \geq \xi(z) > L_{t, \epsilon}$, we may define
\begin{align*}
A := \sum_{k \geq 2} \sum_{\Gamma^\ast_{k}(z, \hat n)} \prod_{0 \leq
i < k} \frac{1}{\tilde{\lambda}_t^{(n)}(z) - \tilde{\xi}\id_{B(z, n)} (y_i)} = o(1)  \,.
\end{align*}
By decomposing each path in $\cup_{k \geq 0} \Gamma_{k, \hat n}$ into a sequence of paths in $\cup_{k \geq 0}\Gamma_{k, \hat n}^\ast$, we get that the right hand side of \cref{eq:pathexpl} is equal to 
$$ \frac{1}{\tilde{\lambda}^{(n)}_t(z)} \sum_{l \geq 0} A^l =  \frac{1}{\tilde{\lambda}^{(n)}_t(z)} \frac{1}{1-A} $$
and so \cref{eq:pathexpl} gives
\begin{align*}
\tilde{\lambda}^{(n)}_t(z) =  \xi(z) +  \tilde{\lambda}^{(n)}_t(z) A  = \xi(z)  + \sum_{k \geq 2} \sum_{\Gamma^\ast_{k}(z, \hat n)} \prod_{0 < i < k} \frac{1}{\tilde{\lambda}_t^{(n)}(z) - \tilde{\xi}\id_{B(z, n)}(y_i)} \,.
\end{align*}
Noticing that $(L_{t, \epsilon} - L_t)^{-(2j-1)} = o(d_t e_t)$, this yields the result by truncating the infinite sum after paths of length $2j$, and since, by \Cref{lem:assep}, eventually $\tilde{\xi} = \xi$ on $B(z, n) \setminus \{z\}$ almost surely.
\end{proof}

\begin{proposition} [Extremal theory for $\tilde{\lambda}_t^{(n)}$;  see {\cite[Section 6]{Astrauskas08}}]
\label{prop:asympt}
For each $n \in \mathbb{N}$ such that $n \leq j$, there exists a scaling function $A^{(n)}_t = a_t + o(1)$ such that, as $t \to \infty$ and for each fixed $x \in \mathbb{R}$, the following are satisfied:
$$ t^d \, \PP   \left(\tilde{\lambda}_t^{(n)}(0) > A^{(n)}_t + x d_t  \right) \to e^{- x}  $$
and
$$  t^d \, \PP  \left(\tilde{\lambda}_t^{(n)}(0) > A^{(n)}_t + x d_t \ , \  \bar{\mathcal{S}}^{(n)}_t(0) \right) \to 0 \,.$$
\end{proposition}
\begin{remark}
In the case $\gamma < 4$, full asymptotics (i.e.\ up to order $d_t$) for $A^{(j)}_t$ can be found in \cite[Section 6]{Astrauskas08}; otherwise, a recurrence formula for $A^{(j)}_t$ is available. Remark also that the same asymptotics hold for each $z \in \mathbb{Z}^d$; we choose the origin for convenience.
\end{remark}

\begin{proof}
\Cref{prop:asympt} is a minor extension of the results in \cite[Section 6]{Astrauskas08}. We prove it in a similar manner to \cite[Theorem 6.3]{Astrauskas08}, by writing the probability as a certain integral and approximating it using Laplace's method. Denote by $f_\xi(x)$ the density function of $\xi(0)$. For a scaling function $C_t \geq a_t $ and a positive field 
$$s^{(n)} := (s_1^{(n)}; s_2^{(n)}) := (\{s_y: y \in \bar{B}_1 \}; \{s_y: y \in \bar{B}_2 \})$$ 
define the function
$$ 
 Q_t^{(n)}(C_t; s^{(n)}) := \sum_{k \geq 2} \sum_{\Gamma^\ast_{k}(0, \hat n)} \prod_{0 < i < k} \frac{1}{C_t - C_t^{q(|y_i|)} s_{y_i} \id_{y_i \in B(0, n)}}  $$
if the sum converges and $Q_t^{(n)}(C_t; s^{(n)}) :=  0$ otherwise, and the functions
\begin{align*}
& R_t^{(n)}(C_t; s^{(n)}) :=  \left(C_t -  Q_t^{(n)}(C_t; s^{(n)}) \right)^\gamma - \sum_{y \in \bar{B}_1} \left( \log f_\xi(C_t^{q(|y|)} s_y ) + \log C_t^{q(|y|)} \right)
\end{align*}
and
\begin{align*}
& P_t^{(n)}(C_t; s^{(n)}) := R_t^{(n)}(C_t; s^{(n)}) - \sum_{y \in \bar{B}_2} \log f_\xi(s_y ) \,.
\end{align*}
To motivate these definitions, consider the first statement of
\Cref{prop:asympt}. Notice that, by \Cref{lem:minmax}, as $t \to \infty$, eventually
$$ \tilde{\lambda}_t^{(n)}(0) > C_t + x d_t \geq a_t + x d_t \implies \xi(0) >
L_{t, \varepsilon} \, .$$
This means that we can apply the path expansion
in \Cref{prop:pathexpforlambda} to $\tilde{\lambda}_t^{(n)}(0)$. Then, since
$\tilde{\lambda}_t^{(n)}(0)$ is strictly increasing in $\xi(0)$, we may write
the probability as the following integrals of $P_t^{(n)}$ and $R_t^{(n)}$ (note the change of variables):
\begin{align}
\label{eq:intrep1} & \PP \left(\tilde{\lambda}_t^{(n)}(0) > C_t \right)  = \int_{\mathbb{R}_+^{|\bar{B}_1 \cup \bar{B}_2|}} \exp \left\{ - P_t^{(n)}(C_t; s^{(n)}) \right\}   \, ds^{(n)} + o(1) \\
\label{eq:intrep2} & \quad = \int_{\mathbb{R}_+^{|\bar{B}_2|}} \prod_{y \in \bar{B}_2} f_\xi(s_y) \left[ \int_{\mathbb{R}_+^{|\bar{B}_1|}} \exp \left\{ - R_t^{(n)}(C_t; s^{(n)}) \right\}   \, ds_1^{(n)} \right] \, ds_2^{(n)}   + o(1) \,.
\end{align}
with the $o(1)$ bound taking care of the contribution from the $s^{(n)}$ for which $Q_t^{(n)}(C_t; s^{(n)})$ does not converge, by \Cref{lem:assep}. 

To approximate these integrals, we state some properties of the functions $P_t^{(n)}$ and $R_t^{(n)}$. Similarly to as in \cite[Section 6]{Astrauskas08}, for a fixed $s_2^{(n)} \in E_2^{(n)}$, the function $R_t^{(n)}(C_t; s^{(n)})$ achieves a minimum at some $s_1^{(n)} \in E_1^{(n)}$. Moreover, for any $s^{(n)} \in E^{(n)}$, the fact that $\eta - 2(\rho+1) < 1 - \gamma$ implies that
\begin{align}
\label{eq:o1}
R_t^{(n)} \left(C_t; s^{(n)} \right)  = R_t^{(n)} \left(C_t; (s_1^{(n)}; 0) \right)  + o(a_t^{-\text{const.}}) 
\end{align}
for a positive constant, where $0$ here denotes the zero vector. The function $R_t^{(n)}(C_t; s^{(n)})$ is also strictly increasing in $C_t$, satisfying
$$ \min_{s_1^{(n)} \in E_1^{(n)}}  R_t^{(n)} \left(C_t; (s_1^{(n)}; 0) \right)  = C_t^\gamma + O(C_t^{\gamma - 2})$$
and, for each $y \in \bar{B}_1^{(n)}$, 
$$ \partial^2_{s_y} R_t^{(n)} | _{(C_t;(1;0))} = O(C_t^{\gamma-2} ) $$
where $1$ here denotes the vector of ones. In particular, this implies that there exists a scaling factor $A^{(n)}_t = a_t + o(1)$ that satisfies
$$  \min_{s_1^{(n)} \in E_1^{(n)} } R_t^{(n)}  \left( A^{(n)}_t; (s_1^{(n)}; 0) \right) + \frac{1}{2}  \sum_{y \in \bar{B}_1} \left[ \log \left( \partial^2_{s_y} R_t^{(n)} |_{(A_t^{(n)};(1;0))} \right) -  \log (2\pi) \right] = a_t^{\gamma}   \,.$$
Remark that if $n = 0$, then $R_t^{(0)} \left(C_t; s^{(0)} \right)  = C_t^\gamma$ and so $A_t^{(0)} = a_t$. Finally, by a similar calculation as in \cite[Lemma 6.8]{Astrauskas08}, if $s^{(n)} \notin E^{(n)}$, then
\begin{align}
\label{eq:gap}
P_t^{(n)}(A^{(n)}_t + x d_t; s^{(n)}) - a_t^\gamma - x > a_t^{\text{c}} \min_{y \in\bar{B}_1 \cup \bar{B}_2} |s_y - 1|^2 
\end{align}
eventually, for some constant $c > 0$. 

Consider now the integral in \cref{eq:intrep2} restricted to the domain $E^{(n)}$. As in \cite[Theorem 6.3]{Astrauskas08}, we may first use \cref{eq:o1} to integrate out over $s_2^{(n)}$, and then apply Laplace's method to approximate the resulting integral over $s_1^{(n)}$:
\begin{align*}
& \int_{E_2^{(n)}} \prod_{y \in \bar{B}_2} f_\xi(s_y) \left[ \int_{E_1^{(n)}} \exp \left\{ -R_t^{(n)}(A^{(n)}_t + xd_t; s^{(n)}) \right\}   \, ds_1^{(n)} \right] \, ds_2^{(n)} \\
& \qquad = \int_{E_1^{(n)}} \exp \left\{ -R_t^{(n)}(A^{(n)}_t + xd_t; (s_1^{(n)}; 0)  \right\}  \, ds_1^{(n)} (1 + o(1)) = t^{-d} e^{-x} (1 + o(1))
\end{align*}
with the last line following from an application of Laplace's method to the integral, noticing that the determinant of the Hessian matrix of $R_t^{(n)}$ with respect to $s_1^{(n)}$, evaluated at a point in $E_1^{(n)}\times\{0\}$, is
asymptotically $ \prod_{y \in \bar{B}_1}  \partial^2_{s_y}R_t^{(n)}
|_{(A_t^{(n)}; (1;0))}  \,. $

Similarly, by \cref{eq:gap}, the integral in \cref{eq:intrep1} over the domain excluding $E^{(n)}$ can be bounded above by
\begin{align*}
t^{-d} e^{-x} \int_{\mathbb{R}_+^{|\bar{B}_1 \cup \bar{B}_2|} \setminus E^{(n)} }  \exp \left\{ - a_t^\text{c}  \min_{y \in\bar{B}_1 \cup \bar{B}_2} |s_y - 1|^2  ) \right\} \, ds^{(n)} =  o( t^{-d} e^{-x} ) \,.
\end{align*}
Together, these two bounds give \Cref{prop:asympt}. 
\end{proof}

\subsection{Constructing the point processes}
We now construct the point processes we shall need to consider. For each $n \in \mathbb{N}$ such that $n \leq j$ and each $z \in V_t$, denote
$$ X^{(n)}_{t, z} := \frac{ \tilde{\lambda}_t^{(n)}(z) - A^{(n)}_{r_t}}{d_{r_t}} \qquad \text{and} \qquad \mathcal{N}^{(n)}_{t} := \sum_{z \in V_t} \id_{(z r_t^{-1}, X^{(n)}_{t, z})}  \,.$$
For each $\tau \in \mathbb{R}$ and $q > 0$ let 
$$ H_\tau^q := \{ (x, y) \in \dot{\mathbb{R}}^d \times (-\infty, \infty] : y \geq q|x| + \tau \}$$
where $\dot{\mathbb{R}}^d$ denotes the one-point compactification of Euclidean space.

\begin{proposition}
For each $n \in \mathbb{N}$ such that $n \leq j$, as $t \to \infty$,
$$ \mathcal{N}^{(n)}_{t}|_{H_\tau^q} \Rightarrow \mathcal{N} \qquad \text{in law}$$
where $\mathcal{N}$ is a point process on $H_\tau^q$ with intensity measure $ \chi(dx, dy) = dx \otimes e^{-y} dy$.
\end{proposition}
\begin{proof}
As in \cite[Lemma 6]{Astrauskas07}, this follows from \Cref{prop:asympt} after checking Leadbetter's mixing conditions modified for random fields (\cite[Theorem 5.7.2]{Leadbetter83}). Again as in \cite[Lemma 6]{Astrauskas07}, since the set $\{\tilde{\lambda}_t^{(n)}(0)\}$ has a dependency range $2n$, it is sufficient to check the following local dependence condition: 
$$ |V_t| \sum_{z : 0 < |z| \le 2n} \PP \left( \tilde{\lambda}_t^{(n)}(0) >
A^{(n)}_{r_t} + x d_{r_t} ,  \tilde{\lambda}_t^{(n)}(z) > A^{(n)}_{r_t} + x d_{r_t} \right) \to 0 $$
as $t \to \infty$, for any $x \in \mathbb{R}$. This is satisfied, since by \Cref{lem:assep} the set $\Pi^{(L_{t})}$ is eventually $2n$-separated almost surely, and so either $\tilde{\lambda}_t^{(n)}(0)$ or $\tilde{\lambda}_t^{(n)}(z)$ is bounded above by $L_{t} < A^{(n)}_{r_t} + x d_{r_t}$ eventually, for any $x$. Observe also that the restriction of $\mathcal{N}^{(n)}_t$ to $H_\tau^q$ ensures that the intensity measure of the limit process $\mathcal{N}$ is such that every relatively compact set has finite measure.
\end{proof}

We transform the point process $\mathcal{N}$ to a new point process involving $\tilde{\Psi}^{(n)}_{t}$.  For technical reasons, we shall need to consider a certain generalisation of the functionals $\tilde{\Psi}^{(n)}_{t}$. So for each $n \in \mathbb{N}$ such that $n \leq j$, $c \in \mathbb{R}$ and sufficiently large $t$, define the functional $ \tilde{\Psi}^{(n)}_{t, c}: V_t \to \mathbb{R}$ by
$$ \tilde{\Psi}^{(n)}_{t, c}(z) := \tilde{\lambda}_t^{(n)}(z) - \frac{|z|}{\gamma t} \log \log t + c \frac{|z|}{t} \,.$$
Let $Z_{t, c}^{(1, n)} := \argmax_{z} \Psi^{(n)}_{t, c}$ and $Z_{t, c}^{(2, n)} := \argmax_{z \neq Z_{t, c}^{(1, n)}} \Psi^{(n)}_{t, c}$. Note that for any $t$ these are well-defined almost surely, since $V_t$ is finite. Further, for each $z \in V_t$ define
$$ Y^{(n)}_{t, c, z} := \frac{ \tilde{\Psi}^{(n)}_{t, c}(z) - A^{(n)}_{r_t} }{d_{r_t}}  \qquad \text{and} \qquad \mathcal{M}^{(n)}_{t, c} := \sum_{z \in V_t} \id_{(z r_t^{-1} , Y^{(n)}_{t, c, z})}  \,.$$ 
Finally, for each $\tau \in \mathbb{R}$ and $\alpha > -1$ let  
$$ \hat{H}_\tau^\alpha := \{ (x, y) \in \dot{\mathbb{R}}^{d+1} : y \geq \alpha|x| + \tau \} \,.$$

\begin{proposition}
\label{prop:pp1}
For each $n \in \mathbb{N}$ such that $n \leq j$ and $c \in \mathbb{R}$, as $t \to \infty$,
$$ \mathcal{M}^{(n)}_{t, c}|_{\hat{H}_\tau^\alpha} \Rightarrow \mathcal{M} \qquad \text{in law} $$
where $\mathcal{M}$ is a point process on $\hat{H}_{\tau}^\alpha$ with intensity measure $ \nu(dx, dy) = dx \otimes e^{- y - |x|} dy$.
\end{proposition}
\begin{remark}
Although we prove \Cref{prop:pp1} for each $c \in \mathbb{R}$, we shall only apply it to $c = 0$ and one other value of $c$ that will be determined in \Cref{corry:upperbound}.
\end{remark}
\begin{proof}
This follows as in \cite[Lemma 3.1]{Sidorova12} (although note that, due to a different choice of $d_t$, the intensity of the point process in \cite[Lemma 3.1]{Sidorova12} differs by a constant). First choose a pair $\alpha'$ and $q$ such that $0 < \alpha' + 1 < q < \alpha + 1$ and notice that 
$$ \mathcal{M}^{(n)}_{t, c}|_{\hat{H}_\tau^\alpha} = \left(
\mathcal{N}^{(n)}_{t}|_{H_\tau^q} \circ K^{-1}_{t, c} \right)|_{\hat{H}_\tau^\alpha}$$
where $K_{t, c} \ : H_\tau^q \to \hat{H}_\tau^{\alpha'}$ is defined by
$$K_{t, c} (x, y) \mapsto 
\begin{cases}
(x, y - (1 + o(1))|x|), & \text{if } x,y \neq \infty \\
\infty & \text{otherwise}
\end{cases} \,.$$
It was proved in \cite[Lemma 2.5]{vanDerHofstad08} that one can pass to the limit simultaneously in the mapping $K_{t, c}$ and the point process $\mathcal{N}^{(n)}_{t}$ to obtain
$$ \mathcal{M}^{(n)}_{t, c}|_{\hat{H}_\tau^\alpha} \Rightarrow \mathcal{M} := \left(
\mathcal{N} \circ K^{-1} \right) |_{\hat{H}_\tau^{\alpha}}$$
in law, where $K \ : H_\tau^q \to \hat{H}_\tau^{\alpha'}$ is defined by
$$K(x, y) \mapsto 
\begin{cases}
(x, y - |x|), & \text{if } x,y \neq \infty \\
\infty & \text{otherwise}
\end{cases} \,.$$
The density of $\mathcal{M}$ is then $\chi \circ K^{-1} = \nu$, restricted to $\hat{H}_\tau^{\alpha}$.
\end{proof}

We now use the point process $\mathcal{M}$ to analyse the joint distribution of the random variables $Z_{t, c}^{(1, n)}$, $Z_{t, c}^{(2, n)}$, $\tilde{\Psi}^{(n)}_{t, c}(Z_{t, c}^{(1, n)})$ and $\tilde{\Psi}^{(n)}_{t, c}(Z_{t, c}^{(2, n)})$. 

\begin{proposition} \label{prop:limit}
For each $n \in \mathbb{N}$ such that $n \leq j$ and each $c \in \mathbb{R}$, as $t \to \infty$
$$ \left(  \frac{Z_{t, c}^{(1, n)}}{r_t}, \frac{Z_{t, c}^{(2, n)}}{r_t},  \frac{ \tilde{\Psi}^{(n)}_{t, c}(Z_{t, c}^{(1, n)}) - A^{(n)}_{r_t}}{d_{r_t}} , \frac{\tilde{\Psi}^{(n)}_{t, c} (Z_{t, c}^{(2, n)}) - A^{(n)}_{r_t}}{d_{r_t}}  \right)    \\
$$
converges in law to a random vector with density
$$ p(x_1, x_2, y_1, y_2) =  \exp \{- (y_1 + y_2) - |x_1| - |x_2|) - 2^d e^{-
y_2} \} \id_{\{y_1 > y_2\}} \,.$$
\end{proposition}
\begin{proof}
\Cref{prop:limit} follows from the point process density in \Cref{prop:pp1} using the same computation as in \cite[Proposition 3.2]{Sidorova12}.
\end{proof}

\subsection{Properties of the localisation site} 
In this subsection we use the results from the previous subsection to analyse
the localisation sites $Z_{t, c}^{(1, j)}$ and $Z_t^{(1, \rho)}$, and in the
process complete the proof of \Cref{thm:main2}.

For each $c \in \mathbb{R}$, introduce the events
$$\mathcal{G}^{(n)}_{t,c} := \{\tilde{\Psi}^{(n)}_{t, c}(Z_{t, c}^{(1, n)}) - \tilde{\Psi}^{(n)}_{t, c}(Z_{t, c}^{(2, n)}) > d_t e_t \} \, , $$
$$ \mathcal{H}_t^{(n)} := \{  r_t f_t < |Z_{t}^{(1, n)}| < r_t g_t \} \quad \text{and} \quad \mathcal{I}_t^{(n)} := \{ \Psi_t^{(n)}(Z_{t}^{(1, n)}) > a_t(1-f_t) \} $$
and the event
\begin{align}
\label{eq:defevent}
\mathcal{E}_{t,c} :=  \mathcal{S}^{(j)}_t(Z_{t}^{(1, j)})  \cap \mathcal{S}_t^{(\rho)}(Z_{t}^{(1, \rho)}) \cap \mathcal{G}^{(j)}_{t,0} \cap \mathcal{G}^{(j)}_{t,c} \cap \mathcal{H}_t^{(j)} \cap \mathcal{I}_t^{(j)}
\end{align}
which act to collect the relevant information that we shall later need.

\begin{corollary} \label{corry:eventj}
For each $c \in \mathbb{R}$, as $t \to \infty$
\begin{align*}
\PP(\mathcal{E}_{t,c}) \to 1 \,.
\end{align*}
\end{corollary}
\begin{proof}
This follows from Propositions~\ref{prop:asympt} and~\ref{prop:limit}, since $A^{(n)}_{r_t} \sim a_t$ and $d_{r_t} \sim d_t$.
\end{proof}

\begin{proposition}
\label{prop:zeqzc}
For any $c \in \mathbb{R}$, on the event $\mathcal{E}_{t,c}$
$$ Z_{t,c}^{(1, j)} = Z_{t}^{(1, j)} $$ holds eventually.
\end{proposition}
\begin{proof}
Assume that $Z^{(1, j)}_{t, c} \neq Z_{t}^{(1, j)}$ and recall that $1/ \log \log t < e_t / g_t$ eventually. On the event $\mathcal{E}_{t,c}$, the statements
$$ \tilde{\Psi}^{(j)}_{t} (Z^{(1, j)}_{t}) - \tilde{\Psi}^{(j)}_{t} (Z_{t, c}^{(1, j)}) > d_t e_t  \quad \text{and} \quad \tilde{\Psi}^{(j)}_{t, c} (Z^{(1, j)}_{t,c}) - \tilde{\Psi}^{(j)}_{t, c} (Z_{t}^{(1, j)}) > d_t e_t $$
and, eventually,
$$|\tilde{\Psi}^{(j)}_{t}(Z_{t}^{(1, j)}) - \tilde{\Psi}^{(j)}_{t, c}(Z_{t}^{(1, j)})| = |c|\frac{|Z_{t}^{(1, j)}|}{t} < \gamma \frac{d_t g_t}{\log \log t} < d_t e_t $$
all hold, giving a contradiction.
\end{proof}

\begin{lemma}
\label{lem:jtorho}
For any $c \in \mathbb{R}$, on the event $\mathcal{E}_{t,c}$
$$ \tilde{\lambda}^{(j)}_{t}(Z_{t}^{(1, j)}) \geq \tilde{\lambda}^{(\rho)}_{t}(Z_{t}^{(1, j)}) \quad \text{and} \quad \tilde{\lambda}^{(j)}_{t}(Z_{t}^{(1, \rho)}) \geq \tilde{\lambda}^{(\rho)}_{t}(Z_{t}^{(1, \rho)})  $$
and
$$ \left( \tilde{\lambda}^{(j)}_{t}(Z_{t}^{(1, j)}) - \tilde{\lambda}^{(\rho)}_{t}(Z_{t}^{(1, j)}) \right) - \left( \tilde{\lambda}^{(j)}_{t}(Z_{t}^{(1, \rho)}) - \tilde{\lambda}^{(\rho)}_{t}(Z_{t}^{(1, \rho)}) \right) < d_t e_t  $$
all hold eventually.
\end{lemma}
\begin{proof}
The first two statements follow from the min-max theorem for the principal eigenvalue, since $j \ge \rho$. For the third statement, we only need consider the case that $j = \rho + 1$. Then, the event $\mathcal{E}_{t,c}$ implies that $\xi(y) < a_t^\eta$ for all $y$ such that $|y - Z_{t}^{(1, j)}| = j$ or $|y - Z_{t}^{(1, \rho)}| = j$. By considering the path expansions in \Cref{prop:pathexpforlambda} for a constant $C > 0$,
\begin{align*}
& \left( \tilde{\lambda}^{(j)}_{t}(Z_{t}^{(1, j)}) - \tilde{\lambda}^{(\rho)}_{t}(Z_{t}^{(1, j)}) \right) - \left( \tilde{\lambda}^{(j)}_{t}(Z_{t}^{(1, \rho)}) - \tilde{\lambda}^{(\rho)}_{t}(Z_{t}^{(1, \rho)}) \right) < \frac{C a_t^{\eta}}{(L_{t, \epsilon} - L_t)^{2j}} < d_t e_t
\end{align*}
eventually, with the last equality holding since $\eta - 2j < 1 - \gamma$.
\end{proof}

\begin{corollary}
\label{corry:zeqz0}
For any $c \in \mathbb{R}$, on the event $\mathcal{E}_{t,c}$
$$ Z_{t}^{(1, j)} = Z_t^{(1, \rho)} \,.$$
\end{corollary}
\begin{proof}
Assume that $Z^{(1, j)}_{t} \neq Z_{t}^{(1, \rho)}$. On the event
$\mathcal{E}_{t,c}$, \Cref{lem:jtorho} implies that
$$\tilde{\Psi}_t^{(j)}(Z^{(1, j)}_{t}) \geq \tilde{\Psi}_t^{(\rho)}(Z^{(1, j)}_{t}) \quad \text{and} \quad \tilde{\Psi}_t^{(j)}(Z^{(1, \rho)}_{t}) \geq \tilde{\Psi}_t^{(\rho)}(Z^{(1, \rho)}_{t}) $$
and 
\begin{align*}
& \left( \tilde{\Psi}^{(j)}_{t}(Z_{t}^{(1, j)}) - \tilde{\Psi}^{(\rho)}_{t}(Z_{t}^{(1, j)}) \right) - \left( \tilde{\Psi}^{(j)}_{t}(Z_{t}^{(1, \rho)}) - \tilde{\Psi}^{(\rho)}_{t}(Z_{t}^{(1, \rho)}) \right) \\
&  \quad = \left( \tilde{\lambda}^{(j)}_{t}(Z_{t}^{(1, j)}) - \tilde{\lambda}^{(\rho)}_{t}(Z_{t}^{(1, j)}) \right) - \left( \tilde{\lambda}^{(j)}_{t}(Z_{t}^{(1, \rho)}) - \tilde{\lambda}^{(\rho)}_{t}(Z_{t}^{(1, \rho)}) \right) < d_t e_t  
\end{align*}
all hold eventually. On the other hand, on the event $\mathcal{E}_{t,c}$ and by the definition of $Z^{(1, \rho)}_{t}$,
$$ \tilde{\Psi}_t^{(j)}(Z^{(1, j)}_{t}) - \tilde{\Psi}_t^{(j)}(Z^{(1, \rho)}_{t}) > d_t e_t  \quad \text{and} \quad \tilde{\Psi}_t^{(\rho)}(Z^{(1, \rho)}_{t}) \geq \tilde{\Psi}_t^{(\rho)}(Z^{(1, j)}_{t})  $$
also hold, giving a contradiction.
\end{proof}

\subsection{Proof of \Cref{thm:main2}}
Fix a constant $c$ as will be defined in \Cref{corry:upperbound}. We prove \Cref{thm:main2} on the event $\mathcal{E}_{t,c}$, since by \Cref{corry:eventj} this event holds with overwhelming probability eventually. Parts~(\ref{thm:main2a}) and~(\ref{thm:main2b}) of \Cref{thm:main2} are implied directly by \Cref{prop:limit} and the definition of the event $\mathcal{E}_{t,c}$. Part~(\ref{thm:main2c}) is proved in an identical manner to the corresponding result in \cite[Section 6]{Sidorova12}. As in \cite[Lemmas 6.2, 6.3]{Sidorova12}, we have that
\begin{align*}
\lim_{t \to \infty}  \PP \left( \left\{ Z_{t + \omega t}^{(1, \rho)} = Z_{t}^{(1, \rho)} \right\} \right)
&= \lim_{n \to \infty} \lim_{t \to \infty} \PP \left( \mathcal{A}(n, \omega, t)
\right)  \\
&= \int_{\mathbb{R}^d \times \mathbb{R}} \exp \left\{  -\nu ( D_\omega(x,y))  \right\} \nu(dx, dy)  < \infty
\end{align*}
where $\mathcal{A}(n, \omega, t)$ is the event
\[
  \mathcal{A}(n, \omega, t) := \Big\{ Y_{t, 0, Z_t^{(1,
  \rho)}}^{(\rho)} \geq -n \Big\}
\bigcap_{z : Y_{t, 0, z}^{(\rho)} \geq -n}
\Big\{
\tilde{\Psi}_{t + \omega t}^{(1, \rho)}(z) \leq
\tilde{\Psi}_{t + \omega t}^{(1, \rho)}(Z_{t}^{(1, \rho)})\Big\}
\]
and $D_{\omega}(x, y)$ is the set
$$D_{\omega}(x, y) := \left\{(\bar{x}, \bar{y}) \in \mathbb{R}^d \times \mathbb{R} : y + \frac{\omega |x|}{1 + \omega} \leq \bar{y} + \frac{\omega |\bar{x}|}{1 + \omega}  \right\} \cup \left(\mathbb{R}^d \times [y, \infty) \right) \,.$$
The random variable $\Theta$ can then be defined by 
\begin{align*}
\qquad \PP(\Theta > \omega) = \lim_{t \to \infty} \PP \left( T^{(\rho)}_t / t \leq \omega \right) = 1 - \lim_{t \to \infty} \PP \left( \left\{ Z_t^{(1, \rho)} = Z_{t + \omega t}^{(1, \rho)} \right\} \right) \,.  \qquad \qed
\end{align*}

\section{Spectral Theory}
\label{sec:spec}
In this section we establish results from spectral theory which we will apply in
\Cref{sec:completion}. The section draws
heavily on~\cite{Astrauskas07} and~\cite{Astrauskas08}. 

\subsection{Notation}
Fix $\varepsilon$, $\varepsilon''$, $\varepsilon'$ and $\theta'$ such that $0 <
\varepsilon'' < \varepsilon < \varepsilon' < \theta < \theta' < 1/2$. 
Along with the usual Hamiltonian 
$\mathcal{H} := \Delta_{V_{t}} + \xi$, define the
$L_{t}$-punctured Hamiltonian
$\tilde{\mathcal{H}} := \Delta_{V_{t}}  +
\tilde{\xi}$, and further, for any $z \in \Pi^{(L_{t})}$, define the `single punctured' and the `single peak' Hamiltonians
$$\mathcal{H}^{(z)} := \mathcal{H} - \xi  \id_{\{z\}} \qquad \text{and} \qquad
\tilde{\mathcal{H}}^{(z)}:= \tilde{\mathcal{H}} +
\xi \id_{\{z\}} \,. $$
Let $\mathcal{G}(\lambda; x, y)$, $\tilde{\mathcal{G}}(\lambda; x, y)$,
$\mathcal{G}^{(z)}(\lambda; x, y)$ and $\tilde{\mathcal{G}}^{(z)}(\lambda; x,
y)$ denote the Green's functions of $\mathcal{H}$, $\tilde{\mathcal{H}}$,
$\mathcal{H}^{(z)}$ and $\tilde{\mathcal{H}}^{(z)}$
respectively.\footnote{We use the following convention for the Green's functions:
  $(\lambda - \mathcal{H}) \mathcal{G}(\lambda; \cdot, y) = \id_{\{y\}}(\cdot)$.}
Let $\tilde{\lambda}_{t}(z)$ denote the principal eigenvalue of
$\tilde{\mathcal{H}}^{(z)}$ and let $\tilde{\lambda}_{t, i}$ denote the $i$'th
largest such $\tilde{\lambda}_{t}(z)$ among $z \in V_t$. Recall also the
definitions of $\lambda_{t, i}$, $\varphi_{t, i}$ and $z_{t, i}$ from
\Cref{sec:outline}.

Moreover, for any $\lambda > L_{t} + 2d$ and $u \in V_t$ define
\begin{align*}
 A(\lambda) &:= \log \frac{\lambda - L_{t}}{2d}
\end{align*}
and
\begin{align*}
  B(\lambda, u) := b(\lambda) \lambda^{-2} \left| \frac{1}{\xi(u)} -
  \tilde{\mathcal{G}}(\lambda; u, u)\right|^{-1}, \quad \text{where} \;
  b(\lambda):= \frac{{(\lambda-L_{t})}^2}{\lambda - L_{t} - 2d} \,.
\end{align*}
Note that $B(\lambda, u) = \infty$ for some $u \in V_t$ and $\lambda$. Finally, introduce the scaling function
\begin{align*}
\delta_t :=  \frac{|V_t|^{\frac{1-2\theta'}{d}}}{\log (1 + (L_{t, \varepsilon'} -
L_{t}) / 2d)
r\left(\Pi^{(L_{t})}\right)} \,.
\end{align*}
\begin{remark}
By \Cref{lem:assep}, almost surely $\delta_t = o(h_t)$.
\end{remark}

\subsection{Ancillary results on eigenvalues}
\begin{proposition}[Correspondence between local and global eigenvalues]
\label{prop:lambda}
The following hold eventually almost surely:
\begin{enumerate}[(a)]
\item
\label{prop:lambdaa}
For all $1 \leq i \leq |V_t|^\varepsilon$,
$$ \tilde{\lambda}_{t, i} = \tilde{\lambda}_{t}(z_{t, i}) $$
\item
\label{prop:lambdab}
$$ \max_{1 \leq i \leq |V_t|^\varepsilon } |\lambda_{t, i} -
\tilde{\lambda}_{t, i}| < \exp \left\{ - |V_t|^{\frac{1 - 2 \theta'}{d}} \right\}  $$
\end{enumerate}
\end{proposition}
\begin{proof}
These are proved in \cite{Astrauskas08}, as a consequence of the third and first statements of Theorem 4.1 respectively (keeping in mind the definition in (4.3) of that paper).
\end{proof}
\begin{remark}
The correspondence in \Cref{prop:lambda} indicates that the $i$'th highest
eigenvalue of $\mathcal{H}$ is closely approximated by the principal eigenvalue
of the `single peak' Hamiltonian $\tilde{\mathcal{H}}^{(z_{t,i})}$. In physical
terms, this can be interpreted as a lack of `resonance' between the regions in
$V_t$ where the potential field $\xi$ is high, i.e.\ the regions which give rise to a high local principal eigenvalue.
\end{remark}
\begin{lemma}[Almost sure asymptotics for eigenvalues]
\label{lem:auxasympt}
The following hold eventually almost surely:
\begin{enumerate}[(a)]
\item
\label{lem:auxasympt0}
$\lambda_{t, i} < L_{t, -\varepsilon'} \qquad \qquad  
\text{for all } 1 \leq i  \leq |V_t|^{\varepsilon} $ 
\item
\label{lem:auxasympt1}
$\lambda_{t, i} > L_{t, \varepsilon'}  \qquad \qquad \ \,
\text{for all } 1 \leq i \leq  |V_t|^{\varepsilon} $ 
\item
\label{lem:auxasympt2}
$\lambda_{t, i} < L_{t, \varepsilon''} \qquad \qquad \
\text{for all } |V_t|^{\varepsilon} \leq i \leq |V_t|$
\item
\label{lem:auxasympt3}
 $\tilde{\lambda}(z_{t, i}) < L_{t, \varepsilon''} \quad  \qquad \text{for all\ }
|V_t|^{\varepsilon} \leq i \leq |V_t|$
\end{enumerate}
\end{lemma}
\begin{proof}
For part~\eqref{lem:auxasympt0}, it is sufficient to show that eventually $\lambda_{t, 1} < L_{t, -\varepsilon'}$, which follows by combining \Cref{lem:minmax} with the asymptotics in \Cref{lem:asforxi}. For parts~\eqref{lem:auxasympt1}-- 
\eqref{lem:auxasympt3}, by the correspondence in \Cref{prop:lambda} it is sufficient to show that eventually
$$ L_{t, \varepsilon'} + 1 < \tilde{\lambda}_{t, \floor{|V_t|^\varepsilon}} < L_{t, \varepsilon''} - 1 \, .$$
As in \Cref{lem:minmax}, by the min-max theorem, for each $z \in V_t$,
$$  \xi(z) \le \tilde{\lambda}_t(z) \le \max\{L_t, \xi(z)\} + 2d \, .$$
The result then follows from the asymptotics in \Cref{lem:asforxi}
\end{proof}

\subsection{Exponential decay of eigenfunctions and upper bound on
eigenfunctions at zero}
In this subsection, we prove that the eigenfunctions $\varphi_{t, i}$
corresponding to the largest eigenvalues of $\mathcal{H}$ eventually localise
with exponential decay away from the localisation site $z_{t, i}$. As a
corollary, we bound the value of these eigenfunctions at the origin. Note that
these results mimic~\cite[Theorem 4.1]{Astrauskas08}, but with tighter control
over the rate of the exponential decay.

\begin{proposition}[See {\cite[Theorem 4.1]{Astrauskas08}}]
\label{prop:expdecay}
Eventually, for all $1 \le i \le |V_{t}|^{\varepsilon}$,
 \begin{align*}
|\varphi_{t, i}(z)| \le 4\left(1 +
\frac{2d}{L_{t, \varepsilon'} - L_{t}}\right) \exp \left( - (1-\delta_t)
\log\left(\frac{\lambda_{t, i} - L_{t}}{2d}\right) |z-z_{t, i}| \right) 
\end{align*}
almost surely.
\end{proposition}
\begin{proof}
\Cref{prop:expdecay} is an application of~\cite[Theorem B.3]{Astrauskas08} with the following notation:
\begin{align*}
L \leftarrow L_{t} \, , \quad  \Pi \leftarrow \Pi^{(L_{t})} \, ,
\quad  h \leftarrow L_{t, \varepsilon'} - L_{t} \, 
\\ \quad  \delta
\leftarrow \delta_t,  \quad \mu \leftarrow \frac{d(1+\theta)}{1-2\theta'} \quad
\text{and} \quad  K \leftarrow |V_t|^{\varepsilon} \,. 
\end{align*}
Note that for this application the assumptions
\cite[(B.25)-(B.29)]{Astrauskas08} are implicitly verified in  the proof of
\cite[Theorem 4.1]{Astrauskas08}, since $\delta_t$ agrees with the $\delta$ used
in that proof.\footnote{Note, however, that the $\delta$ used in the proof
of~\cite[Theorem 4.1]{Astrauskas08} is not explicit, but can be inferred by
jointly considering~\cite[Lemma 4.3]{Astrauskas08}
and~\cite[B.28]{Astrauskas08}. Note also that $\delta$ depends on $\mu$.}
\end{proof}
\begin{corollary}[Exponential decay of eigenfunctions]
\label{corry:expdecay}
Eventually, for all $1 \le i \le |V_{t}|^{\varepsilon}$,
\begin{align*}
\log |\varphi_{t, i}(z)|
\le -\frac{|z - z_{t, i}|}{\gamma} 
(1 - f_t) \log \log t \quad \text{almost surely}\,.
\end{align*}
\end{corollary}
\begin{proof}
For all $1 \le i \le |V_t|^{\varepsilon}$, the asymptotics in
part~(\ref{lem:auxasympt1}) of \Cref{lem:auxasympt} imply that
\begin{align*}
\log (\lambda_{t, i} - L_{t}) > \frac{1}{\gamma} \log \log t +  O(1) 
\end{align*}
almost surely. Applying \Cref{prop:expdecay} we get that, eventually,
\begin{align*}
\log |\varphi_{t, i}(z)|
&\le -\frac{1-\delta_t}{\gamma} |z_{t, i} - z| \log \log t \left( 1 -
\frac{\text{const.}}{\log \log t} \right)
\\ \nonumber
&\le -\frac{|z_{t, i} - z|}{\gamma} (1 - f_t) \log \log t 
\end{align*}
almost surely, since $1 / \log \log t = o(f_t)$ and $\delta_t = o(f_t)$.
\end{proof}
\begin{corollary}[Upper bound on eigenfunctions at zero]
\label{corry:upperbound}
There exists a $c>0$ such that eventually, for all $1 \le i \le |V_t|^{\varepsilon}$,
\begin{align*}
\log |\varphi_{t, i}(0)| \le 
  -\frac{|z_{t, i}|}{\gamma} \log \log t
  + c |z_{t, i}| 
\end{align*}
almost surely.
\end{corollary}
\begin{proof}
Again, applying \Cref{prop:expdecay} and part~(\ref{lem:auxasympt1}) of
\Cref{lem:auxasympt} we have that
\begin{align*}
\log |\varphi_{t, i}(0)|  & \le |z_{t, i}| \left(
- (1 - \delta_t)  
\log (L_{t, \varepsilon'} - L_{t}) + \log (2d) \right) + o(1)
\\ & \le - \frac{|z_{t, i}|}{\gamma} \log \log t + c'
|z_{t, i}| + O(|z_{t, i}| \delta_t \log \log t)
\end{align*}
almost surely, which yields the result since $\delta_t \log \log t = o(1)$, replacing $c'$ with some $c > c'$.
\end{proof}
\subsection{Lower bound on eigenfunctions at zero}
In this subsection, we prove a lower bound on the value of certain
eigenfunctions at zero. We only do this for very specific eigenfunctions, which
satisfy an assumption defined below. It will turn out that the eigenfunction
associated to the localisation site $Z_t^{(1, \rho)}$ satisfies this assumption.

\begin{assumption}
\label{assumpt:A}
Introduce an auxiliary set 
$$H(\lambda) = \left\{ x \in V_t | \tilde{\lambda}_{t}(x) \ge \lambda \right\} \,.$$
An index $i$ satisfies \Cref{assumpt:A} if 
$$ \min_{x \in H(\lambda_{t, i})} |x| > |z_{t, i}|(1+ h_t) \,.$$
\end{assumption}
Before embarking on the proof of a lower bound, we need to introduce some
well-known tools from spectral theory, which are proved, for instance, in
\cite{Astrauskas08}.
\begin{lemma}[{Path expansion over $\Delta_{V_t}$; see~\cite[Lemma A.2]{Astrauskas08}}]
\label{lem:pathExpansion}
Consider the Hamiltonian $\mathcal{H}_{\zeta} := \Delta_{V_t} + \zeta$
on $V_t$, where $\zeta$ is any potential field.
Denote by $\mathcal{G}_{\zeta}(\lambda; x, y)$ the Green's function associated
with $\mathcal{H}_{\zeta}$. Then for any $x, y \in V_{t}$,
\begin{align}
\label{eq:pathExpansion}
\mathcal{G}_{\zeta}(\lambda; x, y) = \sum_{\Gamma(x, y)}^{} \prod_{v \in
  V_{t}}^{}{(\lambda - \zeta(v))}^{-n_{v}(\Gamma(x, y))}
\end{align}
provided the series converges. Here the sum $\sum_{\Gamma(x, y)}^{}$ is taken
over all paths
\begin{align*}
  \Gamma : x = v_{0} \to v_1 \to \dots \to v_m := y \quad \text{in $V_{t}$}
\end{align*}
such that $|v_{i} - v_{i-1}|=1$ for each $1 \le i \le m$ and each $m \in
\mathbb{N}$ (i.e.\ the nearest neighbour paths in $V_{t}$ starting at $x$ and
ending at $y$); $n_{v}(\Gamma(x, y))$ denotes the number of times the path
$\Gamma(x, y)$ visits the site $v \in V_{t}$; $|\Gamma(x, y)| := \sum_{v \in
  V_{t}}^{} n_{v}(\Gamma(x, y)) \ge |x-y|$. Note that if $x=y$ and
 $\Gamma(x, x)=0$, then the corresponding summand in
\cref{eq:pathExpansion} is equal to ${(\lambda - \zeta(x))}^{-1}$.
\end{lemma}
\begin{lemma}[Cluster expansion; see {\cite[Lemma A.1]{Astrauskas08}}]
\label{lem:clusterExpansion}
Fix a non-empty subset ${\Pi \subseteq V_{t}}$.
For a positive field $\zeta$ and any $u \in \Pi$ let
$\mathcal{G}_{\zeta}(\lambda; x, y),
\tilde{\mathcal{G}}_{\zeta}(\lambda; x, y)$ and
$\tilde{\mathcal{G}}^{(u)}_{\zeta}(\lambda; x, y)$
denote the Green's functions of the Hamiltonians
$\mathcal{H}_{\zeta} := \Delta_{V_t} + \zeta, \tilde{\mathcal{H}}_{\zeta} := 
\Delta_{V_t} + \sum_{x \in V_t \setminus \Pi} \zeta(x)\id_{\{x\}}$ and
$\tilde{\mathcal{H}}_{\zeta}^{(u)} := \tilde{\mathcal{H}}_{\zeta} + \zeta
\id_{\{u\}}$ respectively. Then for any $x, y\in V_{t}$,
\begin{align*}
 \mathcal{G}_{\zeta}&(\lambda; x, y) = \tilde{\mathcal{G}}_{\zeta}(\lambda; x, y)
 \\
&+ \sum_{k \in \mathbb{N}}^{} \sum_{\gamma: u_1 \to \dots \to u_k}^{}
    \tilde{\mathcal{G}}_{\zeta}^{(u_1)}(\lambda; x, u_1)\zeta(u_1)
    \left( \prod_{l=2}^{k}
    \tilde{\mathcal{G}}_{\zeta}^{(u_l)}(\lambda; u_{l-1}, u_l) \zeta(u_l) \right)
    \tilde{\mathcal{G}}_{\zeta}(\lambda; u_k, y)
 \end{align*}
 provided that the series converges. Here the sum $\sum_{\gamma}^{}$ is
 taken over all ordered sets
\begin{align*}
\gamma : u_1 \to u_2 \to \dots \to u_k \quad \text{in\ } \Pi
\end{align*}
such that $u_{i-1} \neq u_i$ for each $2 \le i \le k$, having length $|\gamma|=k-1$.
\end{lemma}
Below we show that the series in \Cref{lem:pathExpansion} converges in our
setting, and results in a lower bound on $\varphi_{t, i}(x)$ for certain indices $i$ and sites
$x$. However, to achieve this lower bound, we need to apply the cluster expansion to the
auxiliary set $H(\lambda_{t, i})$ since paths that hit $H(\lambda_{t, i})$ might
contribute negative terms in the path expansion with respect to $\lambda_{t, i}$.
\begin{proposition}
\label{prop:positiveContribution}
For any $\lambda_{t, i}$ such that $1 \leq i \leq |V_t|^{\varepsilon}$, and for any
$u \notin H(\lambda)$ and any $x \in V_{t}$,
\begin{align*}
\tilde{\mathcal{G}}^{(u)}(\lambda_{t, i}; x, u) > 0 \,.
\end{align*}
\end{proposition}
\begin{proof}
 As in \Cref{prop:expdecay},~\cite[Theorem B.3]{Astrauskas08} is valid for all $1 \le i \leq
  |V_t|^{\varepsilon}$. 
  Moreover, from the proof of~\cite[Theorem B.3]{Astrauskas08}, we conclude
  that~\cite[Theorem B.1]{Astrauskas08} is also valid for the same $i$'s. Hence,
  by~\cite[Theorem B.1(ii)]{Astrauskas08}, we have that $\lambda_{t, i} \notin
  \sigma(\mathcal{H}^{(z_{t, i})})$ so we can apply the cluster expansion in
  \Cref{lem:clusterExpansion} with
$\Pi \leftarrow \{u\}$ and $\zeta \leftarrow \xi \id_{V_{t} \setminus
\Pi^{(L_t)} \cup \{u\}}$ to obtain the following resolvent identity:
\begin{align}
\label{eq:resolventIdentity}
\tilde{\mathcal{G}}^{(u)}(\lambda_{t, i}; x, u) =
\tilde{\mathcal{G}}(\lambda_{t, i}; x, u)/ \left( 1 - \xi(u) \tilde{\mathcal{G}}
(\lambda_{t, i}; u, u)\right) \,.
\end{align}
Note that the series in \Cref{lem:pathExpansion} converges for
$\tilde{\mathcal{G}}(\lambda_{t, i}; x, u)$ by \Cref{lem:auxasympt} and hence
the numerator in \cref{eq:resolventIdentity} is positive. It then suffices to show that 
\begin{align}
\label{eq:notInHProdLessThanOne}
u \notin H(\lambda_{t, i}) \implies \xi(u)
\tilde{\mathcal{G}}(\lambda_{t, i}; u, u) < 1 \,.
\end{align}
Recall that, for $u \notin H(\lambda_{t, i})$, we have that $\lambda_{t, i} > \tilde{\lambda}_{t}(u) $. 
Moreover, by~\cite[Remark B.5]{Astrauskas08} we know that $\tilde{\lambda}_{t}(u)$ is the
principal eigenvalue of $\tilde{H}^{(u)}$ if and only if
$\tilde{\lambda}_{t}(u)$ is the maximal solution of the equation
\begin{align*}
\tilde{\mathcal{G}}(\lambda; u, u) = 1 / \xi(u) \,.
\end{align*}
Finally, by \Cref{lem:pathExpansion} it follows that
$\tilde{\mathcal{G}}(\lambda; u, u)$ is monotonically decreasing with $\lambda$. These three facts give us \cref{eq:notInHProdLessThanOne}.
\end{proof}

\begin{proposition}
For all $1 \leq i \le |V_t|^{\varepsilon}$ and any $x \in \xDomain{z_{t, i}}$,
\label{prop:lowerBoundOnG}
\begin{align*}
\tilde{\mathcal{G}}(\lambda_{t, i}; x , z_{t, i}) \ge
\frac{1}{{( \lambda_{t, i} )}^{|z_{t, i} - x|+1}}.
\end{align*}
\end{proposition}
\begin{proof}
\Cref{prop:lowerBoundOnG} follows from
\Cref{lem:pathExpansion} since the sum in \cref{eq:pathExpansion} is
convergent by the asymptotics in \Cref{lem:auxasympt} and the definition of $\tilde{\mathcal{H}}$.
Moreover, every sum along any path is positive, so we may bound the sum in \cref{eq:pathExpansion} by the contribution from the shortest path 
$x \rightarrow \dots \rightarrow z_{t, i}$.
\end{proof}

\begin{proposition}
\label{prop:auxBoundsOnGreens}
     Fix $u \in \Pi^{(L_t)}$ and $v \in \Pi^{(L_t)} \setminus \{u\}$. Then for all $1 \leq i \le |V_t|$ and any $x \in V_t$, the following hold:
  \begin{align*}
    |\mathcal{\tilde{G}}(\lambda_{t, i}; x, u)| &\le \frac{b(\lambda_{t,
    i})}{\lambda_{t, i}(\lambda_{t, i} - L_{t})} e^{-A(\lambda_{t,
    i})|x-u|} \quad \text{and}\\
    |\tilde{\mathcal{G}}^{(u)}(\lambda_{t, i}; v, u)| &\le 
    \frac{B(\lambda_{t, i}, u)}{\xi(u)}
    e^{-A(\lambda_{t, i})|v-u|} \,.
  \end{align*}
\end{proposition}
\begin{proof}
  This follows from \Cref{lem:clusterExpansion} and
  \Cref{lem:pathExpansion}; see~\cite[Lemma B.2]{Astrauskas08}.
\end{proof}
\begin{lemma}
\label{lem:boundOnGreenZ}
For all $1 \le i \le |V_t|^{\varepsilon}$ satisfying Assumption~\ref{assumpt:A},
and any $x \in V_t$ such that $x \in \xDomain{z_{t, i}}$,
\begin{align*}
\mathcal{G}^{(z_{t, i})}(\lambda_{t, i}; x , z_{t, i}) \ge
\tilde{\mathcal{G}}(\lambda_{t, i}; x , z_{t, i}) + o\left(
\tilde{\mathcal{G}}(\lambda_{t, i}; x , z_{t, i}) \right).
\end{align*}
\end{lemma}
\begin{proof}
From the proof of~\cite[Theorem 4.1]{Astrauskas08} it follows that the series in
the cluster expansion in \Cref{lem:clusterExpansion} is convergent with
$\Pi \leftarrow \Pi^{(L_t)} \setminus \{z_{t, i}\}$ 
and $\zeta \leftarrow \xi \id_{V_t \setminus \{z_{t, i} \}}$. Denote
by $u_1 \to \dots \to u_k$ an ordered set consisting of points from a set
$\Pi^{(L_{t})} \setminus \left\{ z_{t, i} \right\}$
\begin{align*}
& \mathcal{G}^{(z_{t, i})}(\lambda_{t, i}; x , z_{t, i}) =
\tilde{\mathcal{G}}(\lambda_{t, i}; x , z_{t, i})   \\
  + & \sum_{k=1}^{\infty} \sum_{\substack{\gamma: u_1 \to \dots \to u_k \\
  \gamma \cap H(\lambda_{t, i}) = \varnothing}}^{}
  \tilde{\mathcal{G}}^{(u_{1})}(\lambda_{t, i}; x , u_1) \xi(u_1)
\left( \prod_{l=2}^{k} \tilde{\mathcal{G}}^{(u_{l})}(\lambda_{t, i}, u_{l-1},
u_{l}) \xi(u_l) \right) \tilde{\mathcal{G}}(\lambda_{t, i}; u_{k}, z_{t, i})
\\
  + & \sum_{k=1}^{\infty} \sum_{\substack{\gamma: u_1 \to \dots \to u_k \\
  \gamma \cap H(\lambda_{t, i}) \neq  \varnothing}}^{}
  \tilde{\mathcal{G}}^{(u_{1})}(\lambda_{t, i}; x , u_1) \xi(u_1)
 \left( \prod_{l=2}^{k} \tilde{\mathcal{G}}^{(u_{l})}(\lambda_{t, i};  u_{l-1},
 u_{l}) \xi(u_l) \right) \tilde{\mathcal{G}}(\lambda_{t, i}; u_{k}, z_{t, i}).
\end{align*}
The first summation on the right hand side is positive by
\Cref{prop:positiveContribution}.
It remains to show that the second summation is negligible, i.e.\ that
\begin{align*}
\Big| \sum_{k=1}^{\infty} & \sum_{\substack{\gamma: u_1 \to \dots \to u_k \\
\gamma \cap H(\lambda_{t, i}) \neq \varnothing}}^{} 
\tilde{\mathcal{G}}^{(u_{1})}(\lambda_{t, i}; x , u_1) \xi(u_1)
  \left( \prod_{l=2}^{k} \tilde{\mathcal{G}}^{(u_{l})}(\lambda_{t, i}; u_{l-1},
  u_{l}) \xi(u_l) \right) \tilde{\mathcal{G}}(\lambda_{t, i}; u_{k}; z_{t, i}) \Big|
  \\ \nonumber 
 &= o\left(\tilde{\mathcal{G}}(\lambda_{t, i}; x , z_{t, i}) \right) \,.
\end{align*}
First apply \Cref{prop:auxBoundsOnGreens}, which gives that this summation is bounded above by
\begin{align*}
\frac{b(\lambda_{t, i})}{\lambda_{t, i}(\lambda_{t, i} - L_{t})} &
\sum_{k \in \mathbb{N}}^{} \sum_{\substack{\gamma: u_1 \to \dots \to
    u_k \\ \gamma \cap H(\lambda_{t, i}) \neq \varnothing}}^{} 
    B(\lambda_{t, i}, u_{1})  e^{-A(\lambda_{t, i})|u_1 - x|}
  \\ & \times
    \left( \prod_{l=2}^{k} B(\lambda_{t, i}, u_{l}) 
    e^{-A(\lambda_{t, i})|u_{l-1} - u_{l}|} \right) 
    e^{-A(\lambda_{t, i})|u_{k}-z_{t, i}|} \,. 
  \end{align*}
Now use the bound $A(\lambda_{t, i})
 |u_{j-1} - u_{j}| > (1-\delta_t) A(\lambda_{t, i}) |u_{j-1} -
  u_{j}| + \delta_t r(\Pi^{(L_{t})})$ for each $j$ in the product,   and let
  $u_l \in H(\lambda_{t, i})$ for the ordered set $\gamma$ in the above
  summation. Applying $|u_1 - x| + |u_1 - u_2| + \cdots + |u_{l-1} - u_{l}| >
  |u_{l} - x|$,
$|u_{l+1} - u_{l+2}| + \cdots + |u_{k} - z_{t, i}| > 0$,
and the fact that the number of ordered sets is bounded above by $|\Pi^{(L_{t})}|^{k}$, we have that the summand is bounded above by
\begin{align*}
\frac{b(\lambda_{t, i})}{\lambda_{t, i}(\lambda_{t, i} - L_{t})} &
e^{-(1-\delta_t)A(\lambda_{t, i}) |u_l-x|} 
\nonumber 
\\ 
& \times \sum_{k \in \mathbb{N}}^{} \left(|\Pi^{(L_{t})}| \max_{u \in
  \Pi^{(L_{t})} \setminus \{z_{t, i}\} } B(\lambda_{t, i}, u) e^{-\delta_t
    A(\lambda_{t, i})r(\Pi^{(L_{t})})}  \right)^{k}
    \nonumber
\\ \le  \frac{b(\lambda_{t, i})}{\lambda_{t, i}(\lambda_{t, i} - L_{t})} &
e^{-(1-\delta_t)A(\lambda_{t, i})|u_{l} - x|} \,.
\end{align*}
Note that the last step is justified by~\cite[B.6]{Astrauskas08}, which is valid since
\cite[Theorem B.1]{Astrauskas08} is applicable, as in
\Cref{prop:positiveContribution}. There are now two cases to consider: (i) $x =
0$; and (ii) $x \in \ballOfExpDecay{z_{t, i}}$.
\par \textit{Case~(i).} Assumption~\ref{assumpt:A} implies that $|u_l| > |z_{t,
i}|(1 + h_t)$, and so by \Cref{prop:lowerBoundOnG} we get that 
\begin{align*}
&\frac{1}{\tilde{\mathcal{G}}(\lambda_{t, i}; x , z_{t, i})} \times
\frac{b(\lambda_{t, i})}{\lambda_{t, i}(\lambda_{t, i} - L_{t})}
e^{-A(\lambda_{t, i})(1-\delta_t)(1 + h_t)|z_{t, i}|}
\\ &= O \left( \text{const.}^{|z_{t, i}|} \left(\frac{2d}{\lambda_{t, i} -
L_{t}}\right)^{|z_{t, i}| (h_t - \delta_t - \delta_t h_t )} \right) 
\\
&= O\left[ \exp \left\{ |z_{t, i}| \left( \log (\text{const.}) 
  - \frac{(h_t - \delta_t - \delta_t h_t)}{\gamma} 
\log \log t \right) \right\} \right] = o(1),
\end{align*}
since $h_t \log \log t \to \infty$ and $\delta_t = o(h_t)$. Note that in the last step we also bounded $\lambda_{t, i} - L_t$ from below using the asymptotics in part~(\ref{lem:auxasympt1}) of
\Cref{lem:auxasympt}.
\par
\par \textit{Case~(ii).} Apply the bound
\begin{align*}
  |u_l - x| \ge |u_l - z_{t, i}| - |z_{t, i} - x| \ge h_t |z_{t, i}| -
  |z_{t, i} - x|
\end{align*}
then, similarly to as in \textit{Case (i)}, by \Cref{prop:lowerBoundOnG} we get that
\begin{align*}
&\frac{1}{\tilde{\mathcal{G}}(\lambda_{t, i}; x , z_{t, i})} \times
\frac{b(\lambda_{t, i})}{\lambda_{t, i}(\lambda_{t, i} - L_{t})}
e^{-A(\lambda_{t, i})(1-\delta_t)(h_t |z_{t, i}| - |z_{t, i} - x|)}
\\ &= O \left( \text{const.}^{|z_{t, i} - x|} \left(\frac{2d}{\lambda_{t, i} -
L_{t}}\right)^{(1 - \delta_t) h_t | z_{t, i}| - (2 - \delta_t)|z_{t, i} - x|} \right) 
\\
&= O \left( \text{const.}^{|z_{t, i}|} \left(\frac{2d}{\lambda_{t, i} -
L_{t}}\right)^{\frac{1}{3} h_t(1 + o(1)) |z_{t, i}|} \right)
= o(1)
\end{align*}
since $|z_{t, i} - x| \leq h_t |z_{t, i}|/3$.
\end{proof}
\begin{proposition}[See {\cite[Theorem B.1(iii)]{Astrauskas08}}]
\label{prop:boundOnGz}
  For all $1 \le i \le |V_t|^{\varepsilon}$ and $x \in V_t$,
  \begin{align*}
    |\mathcal{G}^{(z_{t, i})}(\lambda_{t, i}; x, z_{t, i})| \le 
    \frac{2 b(\lambda_{t, i})}{\lambda_{t, i}(\lambda_{t, i}-L_{t})}
    e^{-(1-\delta_t)A(\lambda_{t, i})|x-z_{t, i}|} \,.
  \end{align*}
\end{proposition}
\begin{proof}
As in \Cref{prop:positiveContribution},~\cite[Theorem B.1]{Astrauskas08} is
valid for all $1 \le i \leq
  |V_t|^{\varepsilon}$.
\end{proof}
\begin{proposition}[Lower bound on eigenfunctions]
\label{prop:lowerbound}
For all $1 \leq i \le |V_t|^{\varepsilon}$ satisfying Assumption~\ref{assumpt:A}
and any $x \in \xDomain{z_{t, i}}$, eventually
\begin{align*}
\log |\varphi_{t, i}(x)| > -\frac{|z_{t, i}-x|}{\gamma} (1 + f_t) \log \log t \,.
\end{align*}  
\end{proposition}
\begin{proof}
Again, as in \Cref{prop:positiveContribution},~\cite[Theorem B.1]{Astrauskas08} is valid for all $1 \le i \leq |V_t|^{\varepsilon}$, and so we have that $\lambda_{t, i} \notin \sigma(\mathcal{H}^{(z_{t, i})})$ and
\begin{align*}
\varphi_{t, i}(x) &= \mathcal{G}^{( z_{t, i} )} \left( \lambda_{t, i}; x, z_{t, i}
\right) \bigg( \sum_{y \in V_{t}}^{} (
\mathcal{G}^{(z_{t, i})}(\lambda_{t, i}; y, z_{t, i}))^{2} \bigg)^{-\frac{1}{2}} \,.
\end{align*}
Note that \Cref{prop:lowerBoundOnG} and \Cref{lem:boundOnGreenZ} imply that
\begin{align*}
\log \mathcal{G}^{(z_{t, i})}(\lambda_{t, i}; x, z_{t, i}) > 
-(|z_{t, i} - x|+1) \log \lambda_{t, i} + o(1) 
\end{align*}
and \Cref{prop:boundOnGz} implies that
\begin{align*}
\bigg( \sum_{y \in V_{t}}^{} ( \mathcal{G}^{(z_{t, i})}(\lambda_{t, i}; y, z_{t,
i}))^{2} \bigg)^{\frac{1}{2}} 
&\le \frac{4b(\lambda_{t, i})}{\lambda_{t, i}(\lambda_{t, i} - L_t)} = \frac{4(\lambda_{t, i} - L_t)}{\lambda_{t, i}(\lambda_{t, i} - L_t - 2d)} \,.
\end{align*}
Combining all the above we conclude that
\begin{align*}
  \log |\varphi_{t, i}(x)| &> -|z_{t, i}-x| \log \lambda_{t, i} + O(1) \,.
\end{align*}
The result then follows from part~(\ref{lem:auxasympt0}) of \Cref{lem:auxasympt}, since $1 / \log \log t = o(f_t)$.
\end{proof}

\begin{corollary}[Lower bound on eigenfunctions at zero]
\label{corry:lowerboundAtZero}
 For any $1 \le i \le |V_t|^{\varepsilon}$ satisfying
  Assumption~\ref{assumpt:A} such that $|z_{t, i}| < r_t g_t$,
  \begin{align*}
   \log |\varphi_{t, i}(0)| > - \frac{|z_{t, i}|}{\gamma} \log \log t 
    + o(t d_t e_t) \,.
  \end{align*}
\end{corollary}
\begin{proof}
This follows from \Cref{prop:lowerbound} since $r_t g_t f_t \log \log t = o (t d_t e_t) $.
\end{proof}

\section{Completing the Proof}
\label{sec:completion}
In this section, we complete the proof of the auxiliary \Cref{thm:aux} using the
results obtained in Sections~\ref{sec:pp} and~\ref{sec:spec}. Throughout this
section we fix the constant $c$ from \Cref{corry:upperbound}, and also fix
constants $0 < \varepsilon'' < \varepsilon < \theta$ as in Sections~\ref{sec:pp} and~\ref{sec:spec}. For notational convenience, abbreviate $z_t^{(j)} := z_{t, i_t^{(j)}}$ for $j = 1,2$. Recall also the event $\mathcal{E}_{t,c}$ from \cref{eq:defevent}, whose probability goes to $1$ by \Cref{corry:eventj}.

\subsection{Ancillary lemmas}
\begin{lemma}[Correspondence with $j$-local eigenvalues]
\label{lem:corres}
Almost surely,
\begin{align*}
\max_{z \in \Pi^{(L_{t, \varepsilon})}} |\tilde{\lambda}_t(z) - \tilde{\lambda}_t^{(j)}(z) | = o(d_t e_t) \,.
\end{align*}
\end{lemma}
\begin{proof}
As in \Cref{prop:pathexpforlambda}, there is a path expansion for $\tilde{\lambda}_t(z)$ with $z \in \Pi^{(L_{t, \varepsilon})}$:
\begin{align}
\label{ex:pathexp}
\tilde{\lambda}_t(z) =  \xi(z) + \sum_{2 \leq k \leq 2j} \sum_{\Gamma^\ast_{k}(z)} \prod_{0 < i < k} \frac{1}{\tilde{\lambda}_t(z) - \xi(y_i)} + o(d_t e_t)  
\end{align}
where $\Gamma^\ast_{k}(z)$ is the set of all length $k$ nearest neighbour paths 
$$ z =: y_0 \to y_1 \to \ldots \to y_{k} := z \quad \text{in } V_t$$
such that $y_i \neq z$ for all $1 < i < k$. Since paths in $\Gamma^\ast_k(z)$ that are not also in $\Gamma^\ast_k(z, j)$ must have length at least $2j+2$, and since $j$ was chosen precisely so that $1/(L_{t, \varepsilon} - L_t)^{2j+1} = o(d_t)$, comparing equation \eqref{ex:pathexp} to the path expansion in \Cref{prop:pathexpforlambda} yields the result.
\end{proof}

\begin{lemma}[Validity of \Cref{assumpt:A}]
\label{lem:assumptA}
On the event $\mathcal{E}_{t,c}$, each 
$$z \in B\left(0, |Z_{t}^{(1, \rho)}|(1 + h_t)\right) \setminus \{Z_{t}^{(1, \rho)}\} $$
satisfies
$$ \tilde{\lambda}_t(Z_{t}^{(1, \rho)}) - \tilde{\lambda}_t(z)  > d_t e_t + o(d_t e_t) \,.$$
\end{lemma}
\begin{proof}
On the event $\mathcal{E}_{t,c}$, and by the correspondence in
\Cref{lem:corres}, for such a $z$,
\begin{align*}
d_t e_t < \tilde{\Psi}^{(j)}_{t}(Z_{t}^{(1, \rho)}) - \tilde{\Psi}_{t}^{(j)}(z) = \tilde{\lambda}_t(Z_{t}^{(1, \rho)}) - \tilde{\lambda}_t(z) + \frac{|z| - |Z_{t}^{(1, \rho)}| }{\gamma t} \log \log t + o(d_t e_t) \,.
\end{align*}
Moreover, for such a $z$, we also have that
$$ \frac{|z|-|Z_{t}^{(1, \rho)}|}{\gamma t} \log \log t < \frac{|Z_{t}^{(1, \rho)}| h_t }{\gamma t} \log \log t < \frac{r_t g_t h_t}{\gamma t} \log \log t = o(d_t e_t) $$
since $g_t h_t = o(e_t)$, which yields the result.
\end{proof}

\begin{lemma}[Application of lower bound]
\label{lem:lowerbound}
On the event $\mathcal{E}_{t,c}$, the following hold:
\begin{enumerate}[(a)]
\item
There exists an index $k_t \le |V_t|^\varepsilon$ such that, eventually, 
$$ z_{t, k_t} = Z_{t}^{(1, \rho)} \, ;$$
\item
Moreover,
$$\Psi_{t}(k_t) > \tilde{\Psi}_{t}^{(j)}(Z_{t}^{(1, \rho)}) + o(d_t e_t) > a_t(1-f_t)  + o(d_t e_t) \,.$$ 
\end{enumerate}
\end{lemma}
\begin{proof} 
On the event $\mathcal{E}_{t,c}$ and from the correspondence in \Cref{lem:corres} we have
$$ \tilde{\lambda}_t(Z_{t}^{(1, \rho)}) > \tilde{\lambda}_t^{(j)}
(Z_{t}^{(1, \rho)}) + o(d_t e_t) > \tilde{\Psi}_{t}^{(j)}(Z_{t}^{(1, \rho)}) + o(d_t e_t) > a_t(1-f_t) + o(d_t e_t) \,.$$
On the other hand, by the asymptotics in part~\eqref{lem:auxasympt3} of \Cref{lem:auxasympt},
$$ \tilde{\lambda}_{t, \floor{|V_t|^\varepsilon}} < L_{t, \varepsilon''} < a_t(1-f_t)  + o(d_t e_t) \,.$$
Hence there exists an index $k_t \leq |V_t|^\varepsilon$ such that $\tilde{\lambda}_{t, k_t} = \tilde{\lambda}_t(Z_{t}^{(1, \rho)}) $
and by the correspondence in \Cref{prop:lambda} this implies that $z_{t, k_t} = Z_{t}^{(1, \rho)}$. By \Cref{lem:assumptA}, it can be seen that
$k_t$ satisfies \Cref{assumpt:A}, and since $|Z_{t}^{(1, \rho)}| < r_t g_t$ on the event $\mathcal{E}_{t,c}$, we may apply the lower bound
in \Cref{corry:lowerboundAtZero} to the index $k_t$. Along with the correspondence in \Cref{lem:corres} we get that
\begin{align*}
\Psi_{t}(k_t) &= \lambda_{t, k_t} + \frac{\log |\varphi_{t, k_t}(0)|}{t} > \tilde{\lambda}_t^{(j)} (Z_{t}^{(1, \rho)}) - \frac{|Z_{t}^{(1, \rho)}|}{\gamma t} \log \log t + o(d_t e_t) \\
&= \tilde{\Psi}_{t}^{(j)}(Z_{t}^{(1, \rho)}) + o(d_t e_t) > a_t(1-f_t) + o(d_t e_t) \,.
\qedhere
\end{align*}
\end{proof}

\begin{lemma}[Application of upper bound]
\label{lem:upperbound}
On the event $\mathcal{E}_{t,c}$, the following hold:
\begin{enumerate}[(a)]
\item
The index $\maxIndex$ satisfies
$$\maxIndex \le |V_{t}|^{\varepsilon} \, ;$$ 
\item
Moreover,
$$\Psi_{t}(\maxIndex) < \tilde{\Psi}^{(j)}_{t, c}(\maxZ) + o(d_t e_t) \,.$$ 
\end{enumerate}
\end{lemma}
\begin{proof}
Combining \Cref{lem:lowerbound} with the event $\mathcal{E}_{t,c}$ implies that 
$$ \lambda_{t, \maxIndex} \geq \Psi_{t}(\maxIndex) \geq \Psi_{t}(k_t) > \tilde{\Psi}^{(j)}_{t}(Z_{t}^{(1, \rho)}) + o(d_t e_t) > a_t(1-f_t) + o(d_t e_t) \,.$$
On the other hand, by the asymptotics in part~\eqref{lem:auxasympt2} of \Cref{lem:auxasympt},
$$ \lambda_{t,\floor{|V_t|^\varepsilon}} < L_{t, \varepsilon''} < a_t(1 - f_t) + o(d_t e_t)$$
and so $\maxIndex \leq |V_t|^\varepsilon$. We may then apply the upper bound in \Cref{corry:upperbound} to the index $\maxIndex$. Combining this with the correspondence in \Cref{lem:corres} we get
\begin{align*}
\Psi_{t}(\maxIndex) &<\tilde{\lambda}_t^{(j)} (\maxZ) - \frac{|\maxZ|}{\gamma t} \log \log t + \frac{c|\maxZ|}{t} \log \log t + o(d_t e_t) \\
&= \tilde{\Psi}^{(j)}_{t, c}(\maxZ) + o(d_t e_t) \,.
\qedhere
\end{align*}
\end{proof}

\begin{lemma}[Correspondence between $\maxZ$ and $Z_{t}^{(1, \rho)}$] 
\label{lem:zeqz}
On the event $\mathcal{E}_{t,c}$,
$$ \maxZ = Z_{t}^{(1, \rho)} $$
eventually.
\end{lemma}
\begin{proof}
Let $k_t$ be as in \Cref{lem:upperbound} and assume that $\maxZ \neq Z^{(1, \rho)}_t = z_{t, k_t}$. Combining the lower bound from \Cref{lem:lowerbound} and the upper bound from \Cref{lem:upperbound}, we have
\begin{align*} 
\tilde{\Psi}^{(j)}_{t, c}(Z_{t}^{(1, \rho)}) & > \tilde{\Psi}^{(j)}_{t, c}(\maxZ) + d_t e_t > \Psi_t(\maxIndex) + d_t e_t + o(d_t e_t) \\
& > \Psi_t(k_t) + d_t e_t + o(d_t e_t) > \tilde{\Psi}^{(j)}_{t} (Z_{t}^{(1, \rho)}) + d_t e_t + o(d_t e_t) \,.
\end{align*}
On the other hand, on the event $\mathcal{E}_{t,c}$
$$|\tilde{\Psi}^{(j)}_{t}(Z_t^{(1, \rho)}) - \tilde{\Psi}^{(j)}_{t, c}(Z_{t}^{(1, \rho)})| = |c|\frac{|Z_{t}^{(1, \rho)}|}{t} \log \log t = o(d_t e_t) $$
giving a contradiction.
\end{proof}

\subsection{Completion of the proof of the auxiliary \Cref{thm:aux}}
We are now in a position to establish the auxiliary \Cref{thm:aux}, which we prove subject to the event $\mathcal{E}_{t,c}$, since this event holds eventually with overwhelming probability. 

First, consider part~\eqref{thm:aux1} of \Cref{thm:aux}. Let $k_t$ be as in \Cref{lem:upperbound}, and remark that \Cref{lem:zeqz} implies that $k_t = \maxIndex$. There are two cases to consider: (i) $\secIndex > |V_t|^\varepsilon$; and (ii) $\secIndex \leq |V_t|^\varepsilon$. In the first case, combine the lower bound in \Cref{lem:lowerbound} with the asymptotics in part~\eqref{lem:auxasympt2} of \Cref{lem:auxasympt} to get that
$$ \Psi_t(\maxIndex) - \Psi_t(\secIndex) > a_t(1-f_t) - L_{t, \varepsilon''} + o(d_t e_t) > d_t e_t$$
which yields the result. In the second case, we may apply the upper bound in \Cref{corry:upperbound} to the index $\secIndex$. Combining this bound with the correspondence in \ \Cref{lem:corres} and the lower bound in \Cref{lem:lowerbound}, we get that eventually
\begin{align*}
\Psi_t(\maxIndex) - \Psi_t(\secIndex) &> \tilde{\Psi}^{(j)}_{t}(\maxZ) - \tilde{\Psi}^{(j)}_{t, c}(\secZ) + o(d_t e_t) \\ &> d_t e_t - |c|\frac{|Z_{t}^{(1, \rho)}|}{t}\log \log t + o(d_t e_t) > d_t e_t + o(d_t e_t) \,.
\end{align*}

Consider now part~\eqref{thm:aux2} of \Cref{thm:aux}. The first statement
follows trivially from the fact that $|Z_t^{(1, \rho)}| < r_t g_t$. The second statement is proved by applying the upper bound on eigenfunctions in
\Cref{corry:expdecay} and the lowerbound on eigenfunctions in
\Cref{prop:lowerbound} to the index $\maxIndex$, which is valid since
$\maxIndex$ satisfies \Cref{assumpt:A}. Then, make the correspondence
in \Cref{lem:zeqz}, and remark that $\kappa_t r_t = o(h_t|Z_t^{(1, \rho)}|)$. The third
statement follows by applying the upperbound on
eigenfunctions in \Cref{corry:expdecay} to the index $\maxIndex$, and summing over all $z \in V_t \setminus B_t$.

%
\bibliography{paper}{}
\bibliographystyle{plain}
\end{document}